\documentclass[a4paper]{article}

\usepackage[english]{babel}
\usepackage[utf8x]{inputenc}
\usepackage[T1]{fontenc}

\usepackage{amsmath,amsthm}
\usepackage{amssymb}
\usepackage{bbm, dsfont}
\usepackage{wasysym}%
\usepackage{color}%
\usepackage{capt-of}

\usepackage[a4paper,top=3cm,bottom=2cm,left=3cm,right=3cm,marginparwidth=1.75cm]{geometry}

\usepackage{amsmath}
\usepackage{graphicx}
\usepackage[colorinlistoftodos]{todonotes}
\usepackage[colorlinks=true, allcolors=blue]{hyperref}

\newtheorem{theorem}{Theorem}
\newtheorem{proposition}{Proposition}
\newtheorem{assumption}{Assumption} 
\newtheorem{lemma}{Lemma}
\newtheorem{corollary}{Corollary}
\newtheorem{remark}{Remark}
\newtheorem{definition}{Definition}

\newcommand{\R}{{\mathbb R}} 
\newcommand{\rzeromas}{r_{0}^{+}} 
\newcommand{\rzeromenos}{r_{0}^{-}} 
\newcommand{\rmumas}{r_{\tildemuI}^{+}} 
\newcommand{\rmumenos}{r_{\tildemuI}^{-}}
\newcommand{\rwmas}{r_{w}^{+}} 
\newcommand{\rwmenos}{r_{w}^{-}} 
\newcommand{\rbar}{\bar r^{+}}
\newcommand{\xssup}{x_S^{\sup}}
\newcommand{\xtsup}{x_T^{\sup}}

\newcommand{\tildemuI}{\mu_I}
\newcommand{\xbar}{\bar x_C}
\newcommand{\xref}{x_{ref}}
\newcommand{\lambdancl}{\lambda_{\rm ncl}}
\newcommand{\xasin}{x_{\rm asymp}}

\def\control{u} 

\usepackage{color}

\title{Optimal control of diseases in prison populations through screening policies of new inmates}

\author{{\sc P. Gajardo$^{1}$, V. Riquelme$^{1}$ and D. Vicencio$^{1}$}\\[2mm]
$^{1}$ Departamento de Matem\'atica, Universidad T\'ecnica Federico Santa Mar\'ia,\\Avenida Espa\~na 1680, Valpara\'iso, Chile\\[2mm]
{\tt pedro.gajardo@usm.cl, victor.riquelmef@usm.cl,
 diego.vicencio@alumnos.usm.cl}\\[2mm]
}

\begin{document}
\maketitle

\begin{abstract}
In this paper, we study an optimal control problem of a communicable disease in a prison population. In order to control the spread of the disease inside a prison, we consider an active case-finding strategy, consisting on screening a proportion of new inmates at the entry point, followed by a treatment depending on the results of this procedure. The control variable consists then in the coverage of the screening applied to new inmates.  The disease dynamics is modeled by a SIS (susceptible-infected-susceptible) model, typically used to represent diseases that do not confer immunity after infection. We determine the optimal strategy that minimizes a combination between the cost of the screening/treatment at the entrance and the cost of maintaining infected individuals inside the prison, in a given time horizon.  Using the Pontryagin Maximum Principle and Hamilton-Jacobi-Bellman equation, among other tools, we provide the complete synthesis of an optimal feedback control, consisting in a bang-bang strategy with at most two switching times and no singular arc trajectory, characterizing different profiles depending on model parameters.

\end{abstract}

\section{Introduction}


Several communicable diseases, such as sexually transmitted infections (STIs), remain a public health problem that is far from being controlled \cite{world2016}.  In prisons, there are far higher prevalences of some diseases  than in the general population, due to, among other reasons, crowded environments, high-risk behaviors such as non protected sexual relations  \cite{kouyoumdjian2012}, and because during imprisonment, the probability of appearance of disease risk factors (e.g., depression, drug use,  etc.) increases. Another factor is related to the deficiencies of the health systems in prisons, which imply the existence of barriers to access to care, which delays diagnoses and prolongs contagion times \cite{belenko2008, khan2009incarceration, MariaetAl2011,kouyoumdjian2012}. This is a general social problem and not only a penitentiary concern, because prisons are acting as reservoirs of diseases, which are transmitted later to the community when inmates are released, or when they are in contact with the outside population (e.g., visitors and prison workers).  Because of this, the World Health Organization (WHO) includes prisoners in the key populations that should be the focus of interventions designed to reduce the burden of  diseases \cite{who2007,world2016}.

Currently, for some diseases, there are proven technological resources that facilitate access to the diagnosis, and subsequent treatment, of communicable diseases in low-resource contexts, such as the existence of low complexity rapid tests, which can be applied by personnel with basic training  \cite{blandford2007,boelaert2007,castillo2020,gianino2007,lee2010,lee2011,pascoe2009,peeling2009}. For STIs, the diagnostic strategies based on rapid tests have the potential to increase access to definitive diagnosis for asymptomatic patients, preventing the development of long-term complications, and cutting the chain of disease transmission in the population \cite{peeling2006}. The above notwithstanding, little is known about the use of these technologies in prison contexts, where the most widespread practice for identifying cases is the passive case-detection, mainly based in symptoms, despite the existence of relatively less expensive alternatives that would allow switching to active case-finding strategies \cite{castillo2020}. 


Our objective in this paper is to provide a theoretical contribution in the explained context, determining  the optimal strategy of active case-finding at the entry point of a prison, considering as (scalar) control variable the coverage in the application of screening (for instance, with rapid tests) to new inmates before coming into contact with the prison population, assuming that all detected cases are satisfactorily treated.  For this purpose, we use one of the simplest epidemiological models, which is the SIS (susceptible-infected-susceptible) model, for representing a disease dynamics inside the prison. We work with this model, because our aim is to provide analytically the complete synthesis of an optimal feedback control, rather than numerically solving the optimization problem, as can be found in the literature, in the context of epidemics control, for  more complex models (see for instance \cite{ejoce2,ejoce3,ejoce1}).  An analytical solution has the advantage of providing useful information about properties of the optimal strategies, and on the impact of some parameters in these strategies, which can then be tested in more realistic models. SIS-type models represent diseases that do not confer immunity against reinfection (e.g., meningitis, gonorrhea, etc.) \cite{BrauerCCC2013}. In the prison context, we assume that the total population is fixed. This assumption can be explained by the overcrowding of many of them, so as soon as a vacancy is released, it is immediately filled by a new inmate.


The cost to be minimized is a combination between the cost of the screening/treatment at the entrance, and the cost of maintaining infected individuals inside the prison, in a given horizon time.  This leads to study an optimal control problem governed by a one-dimensional ordinary differential equation, which is affine in the control variable (screening coverage at the entry point) and quadratic in the state variable (proportion of infected inmates). Thus, the problem does not fit in the classical linear-quadratic formulation (e.g., \cite{campbell1976,Rockafellar1987,sontag1998}). Although it is a one dimensional problem, even affine in the control, the obtained formulation requires a careful analysis, because \emph{a priori} it can exhibit technicalities such as several switching points for optimal trajectories, as well as the appearance of turnpike and anti-turnpike singular arcs. 

Reformulating the problem, one can write the objective function to be minimized, as a purely state-dependent cost (we do not do that). In this framework, when the integrand of the objective function is convex with an isolated minimum, as in the seminal paper \cite{Clark1979} (in the context of natural resource management), the solution is the  most rapid approach path (MRAP)  to the state which minimizes the integrand. This kind of model has been studied later in \cite{RapaportCartigny2004,RapaportCartigny2005}, analyzing  calculus of variations problems with infinite time horizon, whose  integrands depend linearly on the state variable and on its velocity. Using  the Hamilton-Jacobi-Bellman equation,  in \cite{RapaportCartigny2005} the authors  characterize the optimal solutions as the MRAP to some solution of the corresponding Euler-Lagrange equation. When there is  no feasible control to remain in a turnpike singular arc (as needed in a MRAP solution), the problem has been studied recently by \cite{BayenetAl2015} and \cite{dmitruk2020}. All the mentioned papers show the interest for characterizing analytically the solutions of one-dimensional optimal control problems. Nevertheless, the problem formulation that we obtain, for the control of diseases in prison populations through screening policies of new inmates, to the best of our knowledge, does not fit in a framework already studied and, therefore,  it has not been solved analytically.  For this problem, we show the existence of an anti-turnpike singular arc, so there is not a singular arc optimal trajectory. Also, as a direct consequence of the Pontryagin Maximum Principle \cite{pontryagin}, we show that there are at most two switching instants. Knowing that optimal solutions are of bang-bang type, and knowing the number of maximal switching points, the synthesis of an optimal feedback control is obtained  dividing the  state-time plane with the switching curves, and proving that in each region, the  trajectories associated to (bang) constant controls, are solutions of the corresponding Hamilton-Jacobi-Bellman (HJB) equation \cite{BarCapBook}. To this purpose, it is necessary to compute the derivatives of value function with respect to the switching  points (time and state) through  a formula established by one of us in \cite{victor:ode}. Once the HJB equation is verified in each region, we conclude using the result given in \cite[Corollary 7.3.4]{BressanPiccoli2007}, recalled in this paper,  where is proven that the entire trajectory is an optimal solution of the original problem.

The paper is organized as follows. In Section \ref{sec:model} we describe the model and formulate the optimal control problem, to continue in Section \ref{sec:assum-prelim} presenting some assumptions and preliminary results. In Section \ref{sec:pontryagin}, using the  Pontryagin Maximum Principle and studying feasible curves, we deduce a series of properties for optimal trajectories. The complete synthesis of an optimal control, in feedback form, is obtained in Section \ref{sec:sintesis}. Finally, some concluding remarks are stated in Section \ref{sec:conclusions}.


\section{Communicable diseases with direct transmission and constant population size (SIS model)}\label{sec:model}

Let us consider the modeling framework proposed in \cite{BrauerCCC2013} with a constant human population size $N >0$. Define $S(t) \geq 0$ and $I(t) \geq 0$ as the quantities of individuals that belong to the susceptible and infective classes, so that $S(t) + I(t)=N$ for all $t \geq 0$. Suppose also that people enter the prison at a constant rate $\mu N$, a proportion $p_I \in [0,1]$ of which are infected and $1-p_I$ are susceptible to the infection. People also exits the prison at the same entry rate. We suppose that the spread of the disease inside the prison can be controlled by the application of a screening procedure (test) to a fraction $\control \in [0,1]$ of the population that enters the prison. This test has a sensitivity (true positive rate) of $1-\eta$, and a specificity (true negative rate) of $1-\delta$, with $\eta, ~\delta \in [0,1]$. After the test, an immediate treatment to the people whose tests results are positive is performed. The probability of successful recovery of an infected individual is $\pi\in[0,1]$. A scheme of this process is shown in Figure \ref{fig:esquema_entrada}.

\begin{center}
\includegraphics[scale=0.04]{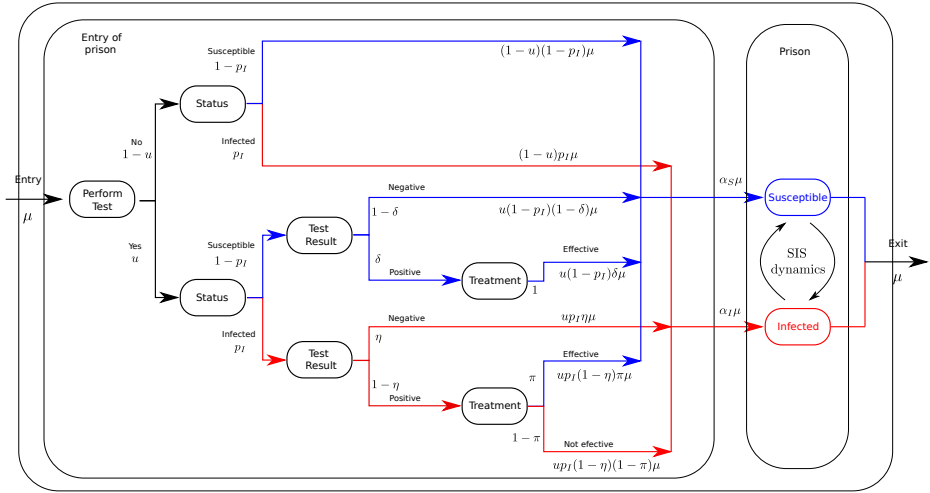}
\captionof{figure}{Scheme of test and treatment at the entry of the prison.}\label{fig:esquema_entrada}
\end{center}

The proportion $\alpha_I \in [0,1]$ of infected people effectively entering the prison after the screening-treatment procedure is composed by
\begin{itemize}\itemsep0em
\item Infected people, tested, detected by test, for which the treatment does not work;
\item Infected people, tested, not detected by the test;
\item Infected people, not tested.
\end{itemize}
Then, 
\begin{equation*}
\alpha_I \,=\, up_I(1-\eta)(1-\pi) + up_I\eta + (1-u)p_I \,=\, p_I(1-u(1-\eta)\pi) .
\end{equation*}

In order to work with the proportions of susceptible and infected populations inside the prison we consider $s(t):=S(t)/N\in[0,1]$ and $x(t)=I(t)/N\in[0,1]$. If the disease transmission occurs at the per-capita contact rate $ \beta > 0$ and infected individuals recover from the disease at a rate of $\gamma >0$ with no immunity, then the disease transmission dynamics can be described by the following SIS system of two differential equations:
\begin{equation}\label{eq:sys}
\left\{\,\begin{split}
\dot{s} & = (1-\alpha_I)\mu  -  \beta s \: x + \gamma x - \mu s, \\
\dot{x} & = \alpha_I\mu+ \beta s \: x -  \gamma x - \mu x.
\end{split}\right.
\end{equation}

\noindent Notice that, under dynamics \eqref{eq:sys}, the set $\{(x,s)\in\R^2_+\,|\, x+s = 1\}$ is positively invariant. Since $x(0)+s(0)=1$, we have the equality $s(t)=1-x(t)$ for all $t\geq0$, thus the above system can be reduced to
\begin{equation*}
\dot x 
=\, \mu_I-\tilde\mu_Iu + \beta x(1-x)- (\gamma + \mu)x ,
\end{equation*}
where $\mu_I:=p_I \mu$ is the rate of infected people entering from the outside, and $\tilde\mu_I := (1-\eta)\pi\mu_I\leq\mu_I$ is the rate of infected people entering the prison that was correctly identified by the test and successfully treated. The dynamics of $x$ can be written in a simpler form, as 
\begin{equation}\label{eq:SIS-logistic-2}
\dot x = \mu_I- w + f(x),
\end{equation}
where $f(x):=(\beta-(\gamma+\mu))x-\beta x^2=Bx-Ax^2$ is a concave quadratic function, with $A:=\beta$, $B:=\beta-(\gamma+\mu)$, and the new control $w=\tilde\mu_I u\in[0,\tilde\mu_I]$. We keep this notation from now on.

We are interested in the minimization of  the cost associated to the screening procedure at the entry point, the cost associated to the treatment of detected individuals and the cost of maintaining infected people  in the prison (associated also to detection and treatment) during a fixed time period $T>0$. 

The rate of persons tested at time $t$ (at the entry point) is $\mu u(t)$.  The rate $r(t)$ of treated people at time $t$, due to a positive test at the entry point is composed by
\begin{itemize}\itemsep0em
\item Infected people, tested, detected by test as infected;
\item Susceptible people, tested, detected by test as infected.
\end{itemize}
Then, 
\begin{equation*}
r(t) \,=\, u(t)\mu(1-\eta)p_I + u(t)\mu \delta(1-p_I) \,=\, u(t)\mu \left[(1-\eta)p_I + \delta(1-p_I)\right].
\end{equation*}

Let $C_D,~C_T,~C_I \in (0,+\infty)$ be the unitary costs of test (at the entrance), treatment (produced by a detection in the entrance) and of infected individuals in the prison, respectively. Then, the cost functional is given by
\begin{equation*}
\begin{split}
J(u(\cdot);x_0,t_0) =&\, \int_{t_0}^T \left[C_D u(t)\mu + C_T r(t) + C_I x(t)\right]dt\\
= &\, C_I\int_{t_0}^T  \left[\left(\frac{C_D}{C_I}\mu+ \frac{C_T}{C_I}\mu \left[(1-\eta)p_I + \delta(1-p_I)\right]\right) u(t)  + x(t)\right]dt,
\end{split}
\end{equation*}
with $x(t)$ the solution of \eqref{eq:SIS-logistic-2} with initial condition $x(t_0)=x_0$ at initial time $t_0 \in [0,T)$, and control $u(\cdot)$. In the previous expression we can omit the multiplicative constant $C_I$, and rename $C:=\left(\frac{\mu}{\tilde \mu_I}\right)\left(\frac{C_D}{C_I}\mu+ \frac{C_T}{C_I}\mu \left[(1-\eta)p_I + \delta(1-p_I)\right]\right)$. Thus, we consider a simpler expression for the cost functional, given in terms of the control $w(\cdot)$, by:
\begin{equation*}\label{eq:funcional_costo}
J(w(\cdot);x_0,t_0) = \int_{t_0}^T \left[C w(t) + x^{w(\cdot)}(t;x_0,t_0)\right]dt
\end{equation*}
where $x^{w(\cdot)}(\cdot;x_0,t_0)$ satisfies
\begin{equation}\label{eq:SIS-logistic-3}
\left\{\quad\begin{split}
\dot x(t) =&\,  \mu_I- w(t) + f(x(t)) \quad \mbox{ a.e. }t\in[t_0,T],\\
x(t_0) =&\, x_0,\\
w(t)\in&\,[0,\tilde\mu_I]\quad \mbox{ a.e. } t\in[t_0,T].
\end{split}\right.
\end{equation}

The optimization problem is, for a given initial condition $x_0\in[0,1]$ at initial time $t_0\in[0,T)$,
\begin{equation}\label{eq:problem}
V(x_0,t_0) := \min \left\{ J(w(\cdot);x_0,t_0)\,|\, (x(\cdot),w(\cdot)) \mbox{ satisfies }\eqref{eq:SIS-logistic-3},\, w(\cdot) \mbox{ measurable} \right\}.
\end{equation}

\section{Assumptions and preliminaries}\label{sec:assum-prelim}

In this section we study some preliminary concepts regarding the model. It is immediate to notice that for any given measurable function $w:[t_0,T]\rightarrow [0,\tilde\mu_I]$, the dynamics in \eqref{eq:SIS-logistic-2} is locally Lipschitz (c.f. \cite{peypouquet2015convex}) for $x\in[0,1]$, and thus, by application of Caratheodory's theorem \cite{browder1998ordinary,coddington1955theory}, existence and uniqueness of solutions of \eqref{eq:SIS-logistic-2} can be guaranteed for any initial condition $x(t_0) \in [0,1]$. Since we are dealing with the proportion of the infective population, we have the following lemma:
\begin{lemma}
The interval $[0,1]$ is invariant under dynamics \eqref{eq:SIS-logistic-2} for any control function $w:[t_0,T]\rightarrow [0,\tilde\mu_I]$.
\end{lemma}

\begin{proof}
It suffices to show that the dynamics points inwards the interval $[0,1]$ at the endpoints of this interval. Notice that $\dot x|_{x=0} = \mu_I-w+f(0)=\mu_I-w\geq0$, and $\dot x|_{x=1} = \mu_I-w+f(1)=\mu_I-w-A+B=-\mu(1-\alpha_I)-w-\gamma<0$ for any $w\in[0,\tildemuI]$, which proves the lemma.
\end{proof}

The steady states of the dynamics \eqref{eq:SIS-logistic-2}, given a constant control $w(\cdot) \equiv w\in[0,\tilde\mu_I]$, are
\begin{equation}\label{eq:raices_w}
r^{+}_w := \frac{B + \sqrt{B^2+4A(\mu_I-w)}}{2A}   \quad,\quad   r^{-}_w := \frac{B - \sqrt{B^2+4A(\mu_I-w)}}{2A}.
\end{equation}

We notice that the differential equation \eqref{eq:SIS-logistic-2} corresponds to a logistic equation translated by a positive constant. In the case when $\tilde\mu_I<\mu_I$, since $w \le \tilde \mu_I$, we have $r^{-}_w<0$ and $0<r^{+}_w<1$. It is easy to check that for any such values, the steady state $r^{+}_w$ is globally asymptotically stable for any initial condition $x(t_0)=x_0\in [0,1]$, using the constant control $w(\cdot)\equiv w$. In this particular case, there is no disease-free steady state, and thus the concept of the basic reproduction number does not apply.

In the case when it is possible to have $w=\tilde\mu_I=\mu_I$, that is, considering a perfect sensitivity test ($\eta=0$), with perfect recovery after treatment ($\pi=1$), and performing a full screening ($w=\tilde\mu_I=\mu_I$), the system does not receive infected people. We have 
\begin{equation}\label{eq:rmui}
\rmumenos = 0,\quad \rmumas = \frac{B}{A} = 1-\frac{\gamma+\mu}{\beta}.
\end{equation}
Thus, with $w(\cdot)\equiv \mu_I$, the model \eqref{eq:SIS-logistic-2} corresponds to a standard logistic equation, and it is possible to define $\mathcal{R}_0:=\beta/(\gamma+\mu)$ the {\it basic reproduction number} of the system \cite{BrauerCCC2013}. From now on, we make the following assumptions:

\begin{assumption}
The used test has perfect sensitivity  ($\eta=0$) and the treatment for detected infected people is fully effective ($\pi=1$). In other words, $\tilde\mu_I = \mu_I$. 
\end{assumption}

\begin{assumption}
The basic reproduction number satisfies $\mathcal{R}_0 = \beta/(\gamma+\mu)>1$. 
\end{assumption}

The hypothesis $\mathcal{R}_0>1$ is equivalent to $B=\beta-(\gamma+\mu)>0$. This implies that the so called {\it endemic steady state} is positive (exists) and is  globally asymptotically stable (applying $w\equiv \mu_I$) for all initial conditions $x(t_0)=x_0\in(0,1]$. In turn, the so called {\it disease-free steady state} $r^{-}_{\mu_I}=0$ is unstable for all initial conditions $x(t_0)\in(0,1]$. Hence, under the assumption  $\mathcal{R}_0>1$, even applying the maximal control $w\equiv \mu_I$, the disease does not vanish from the prison population. We think this is the most interesting case to be studied. When $\mathcal{R}_0< 1$, applying the maximal control, for a sufficiently long time, will imply that the number of infected people in the prison converges to zero.

We also observe that the function $w\mapsto \rwmas$ is decreasing (a greater amplitude of screening leads to a lower proportion of infected people at equilibrium). In particular, $r_{0}^{+}>r_{\mu_I}^{+}$. In what follows, we write 
\begin{equation}\label{eq:rzero}
\rzeromenos = \frac{B-\sqrt{\Delta}}{2A},\quad \rzeromas = \frac{B+\sqrt{\Delta}}{2A}, \quad \mbox{ with }\quad \Delta:= B^2+4A\mu_I.
\end{equation} 

Moreover, if $w(t) \equiv w\in[0,\mu_I]$ is a constant control, \eqref{eq:SIS-logistic-3} can be solved by the partial fractions method, and its solution $x^w(t;x_0,t_0)$ is given by
\begin{equation}\label{eq:solucion_logistica}
x^w(t;x_0,t_0) = \frac{r_w^{+}(x_0-r_w^{-}) - e^{-A(r_w^{+}-r_w^{-})(t-t_0)}r_w^{-}(x_0-r_w^{+}) }{ (x_0-r_w^{-}) - e^{-A(r_w^{+}-r_w^{-})(t-t_0)}(x_0-r_w^{+})  }, 
\end{equation}
or, in terms of the time, 
\begin{equation}\label{eq:tiempo_logistica}
t-t_0 = \frac{1}{A(r_w^{+}-r_w^{-})}\log\left| \frac{(x^w(t;x_0,t_0)-r_w^{-})(x_0-r_w^{+}) }{(x^w(t;x_0,t_0)-r_w^{+})(x_0-r_w^{-})} \right|.
\end{equation}

In Section \ref{sec:switching} we take advantage of these expressions.

\section{Study of the optimal control problem}\label{sec:pontryagin}

\subsection{Pontryagin extremals and switches}\label{subsec:pontryagin}

We can guarantee the existence of solutions of the problem \eqref{eq:problem}, via Filippov's theorem \cite{clarke}.  Let us define the Hamiltonian of problem \eqref{eq:problem} by
\begin{equation}\label{eq:hamiltonian}
H(x,\lambda,\lambda^0;w) \, := \, \lambda^0(Cw+x)+\lambda(\mu_I-w + f(x)) \, = \, (\lambda^0 C-\lambda)w + \lambda^0 x +\lambda(\mu_I+f(x)).
\end{equation}
If $w(\cdot)$ is a solution of problem \eqref{eq:problem}, with associated optimal trajectory $x(\cdot)$, then there exists a non null Pontryagin adjoint state $(\lambda^0,\lambda(\cdot))$, with $\lambda^0\in\{0,1\}$, and $\lambda(\cdot)$ absolutely continuous, which satisfies the equation \cite{clarke}
\begin{equation}\label{eq:eq_adjunta}
\dot\lambda(t) = -\lambda^0-\lambda f'(x(t)) \quad \mbox{ a.e. }\, t_0<t<T,\quad \lambda(T)=0,
\end{equation} 
such that
\begin{equation}\label{eq:minham}
H(x(t),\lambda(t),\lambda^0;w(t))=\min_{w\in[0,\tildemuI]} H(x(t),\lambda(t),\lambda^0;w),\quad  \mbox{ a.e. } t\in[t_0,T].
\end{equation} 

In what follows, by Pontryagin extremal (or extremal trajectory) we refer to a tuple $(x(\cdot),\lambda(\cdot),\lambda^0)$ satisfying \eqref{eq:SIS-logistic-3}, \eqref{eq:eq_adjunta}, and \eqref{eq:minham}, for some control $w(\cdot)$, with $\lambda^0\in\{0,1\}$ and $(\lambda^0,\lambda(\cdot))$ not null.

\begin{proposition}\label{prop:miscelaneos_Pontryagin}
Let $(x(\cdot),\lambda(\cdot),\lambda^0)$ be a Pontryagin extremal, with associated control $w(\cdot)$. Thus, 
\begin{enumerate}
\item The existence of abnormal trajectories is discarded. Thus, $\lambda^0=1$. \label{prop:no_abn}

\item The adjoint state satisfies $\lambda(t)>0$ for every $t \in [t_0,T)$.\label{prop:lambda_pos}

\item There exists $t^{\star} \in [t_0,T)$ such that for all $t\in(t^{\star},T]$, $w(t)=0$.

\item There are no singular arcs. \label{prop:no_singular_arc}
\end{enumerate}
\end{proposition}

\begin{proof}
\begin{enumerate}
\item If $\lambda^0=0$, the only possible absolutely continuous solution of \eqref{eq:eq_adjunta} is $\lambda(\cdot)\equiv0$, which cannot happen, since $(\lambda(\cdot),\lambda^0)$ cannot be null. Thus, $\lambda^0\neq0$, from which we conclude that $\lambda^0=1$.

\item \if{The locus in the state space $(x,\lambda)$ such that $\dot\lambda=0$, is the set $\{(x,\lambdancl(x))\,|\, x\in\R \}$, where
\begin{equation}\label{eq:lambda_nulclina}
\lambdancl(x):=-1/f'(x)=-1/(B-2Ax).
\end{equation} 
This function has a vertical asymptote at the unique root of $f'(\cdot)$, which is $\xasin:=B/(2A)$, and it satisfies $\lambdancl(x)>0$ ($f'(x)<0$) for $x>\xasin$, and $\lambdancl(x)<0$ ($f'(x)>0$) for $x<\xasin$. Then,
\begin{itemize}\itemsep0em
\item For $(x,\lambda)$ such that $x>\xasin$ and $\lambda<\lambdancl(x)$, we have $\dot\lambda=-1-\lambda f'(x)<0$. 
\item For $(x,\lambda)$ such that $x<\xasin$ and $\lambda>\lambdancl(x)$, we also have $\dot\lambda=-1-\lambda f'(x)<0$. 
\end{itemize}}\fi

If there exists a time $t^{\dagger}<T$ such that $\lambda(t^{\dagger})=0$, then $\dot\lambda(t^{\dagger})<0$. This implies that $\lambda(t)<0$ for all $t>t^{\dagger}$, contradicting the transversality condition $\lambda(T)=0$. Thus, necessarily $\lambda(t)>0$ for all $t<T$.

\item Since the Hamiltonian is linear with respect to $w$, and $\lambda^0=1$ (by point \ref{prop:no_abn} of this proposition), the minimization of $H(\cdot)$ with respect to the control variable is characterized by the switching function
\begin{equation}\label{eq:sfpsi}
\psi(x,\lambda) := C-\lambda.
\end{equation}
Thus, the control that achieves the minimization of the Hamiltonian along the extremal trajectory can be written as a function of the state variables as $w(t)=\tilde w(x(t),\lambda(t))$, with
\begin{equation}\label{eq:control_min_ham}
\tilde w(x,\lambda)=\left\{\quad\begin{split}
0, \quad \mbox{ if } \psi(x,\lambda)>0 ,\quad \mbox{ equivalently, if } \lambda<C,\\
\tildemuI, \quad \mbox{ if } \psi(x,\lambda)<0 ,\quad \mbox{ equivalently, if } \lambda>C.
\end{split}\right.
\end{equation}

Since the final condition of the adjoint state is $\lambda(T)=0$, with $\lambda(\cdot)$ continuous, there exists $t^{\star}<T$ such that for all $t\in(t^{\star},T]$ we have $\lambda(t)<C$, which is equivalent to $\psi(x(t),\lambda(t))=C-\lambda(t)>0$. Thus, according to \eqref{eq:control_min_ham}, for all $t\in(t^{\star},T]$, $w(t)=0$.

\item The Hamiltonian \eqref{eq:hamiltonian} is linear in the control variable $w$. Suppose that the extremal trajectory and control are in a singular arc. Since $\lambda(t) > 0$ for $t\in[t_0,T)$ (point \ref{prop:lambda_pos}. of this proposition) and $f(\cdot)$ is strictly concave ($f''(x)<0$, for all $x\in\R$), we have
\begin{equation*}
(-1)\frac{\partial}{\partial w}\left[\frac{d^{2}}{dt^{2}} \frac{\partial H}{\partial w}\right] = \lambda f''(x)  < 0,
\end{equation*}
which contradicts the Legendre-Clebsch criteria \cite{kelley,Goh}, that states that along a singular arc,
\begin{equation*}
(-1) \frac{\partial}{\partial w}\left[\frac{d^{2}}{dt^{2}} \frac{\partial H}{\partial w}\right]  \geq 0.
\end{equation*}
Therefore, there cannot exist singular arcs along any extremal trajectory.
\end{enumerate}
\end{proof}

From now on, since $\lambda^0=1$, we omit its dependency in $H(\cdot)$ defined by \eqref{eq:hamiltonian}.\\

We can generate a field of extremals via backward integration. Indeed, for every prescribed final state $x_T\in[0,1]$, we can find the extremal trajectory $(x(\cdot),\lambda(\cdot))$, with associated control $w(\cdot)$ (given in feedback form by \eqref{eq:control_min_ham}, in the state space $(x,\lambda)$), such that $(x(T),\lambda(T))=(x_T,0)$. Now, since the system is time-autonomous, Pontryagin's principle states the constancy of the Hamiltonian along the extremal trajectories. Thus, if $(x(\cdot),\lambda(\cdot))$ is a Pontryagin extremal with associated control $w(\cdot)$, that passes through a point $(\xref,\lambda_{ref})$, then $H(x(t),\lambda(t);w(t)) = H(\xref,\lambda_{ref};\tilde w(\xref,\lambda_{ref}))$ for all $t\in[t_0,T]$ (with $\tilde w(\cdot)$ as in \eqref{eq:control_min_ham}). Then, an extremal trajectory with final condition $x(T)=x_T$ satisfies $H(x(t),\lambda(t);w(t)) = H(x_T,0;\tilde w(x_T,0)) = x_T$, for all $~t \in [t_0,T]$.

Using the constancy of the Hamiltonian, the following proposition shows that every optimal solution has at most two switches.

\begin{proposition}\label{prop:switches}
Let $(x(\cdot),\lambda(\cdot))$ be a Pontryagin extremal. Then, there exist at most two switching states for $x(\cdot)$. Moreover, there exist at most two switches along the trajectory. 
\end{proposition}

\begin{proof}
Suppose the final condition of the extremal trajectory $x(\cdot)$ is $x_T$. Then, if $t^{\dagger}$ is a switching time, we know that $\lambda(t^{\dagger})=C$, and the associated switching state $x^{\dagger}=x(t^{\dagger})$ satisfies
\begin{equation*}
H(x(t^{\dagger}),\lambda(t^{\dagger});w(t^{\dagger})) = x^{\dagger} + C(\mu_I+f(x^{\dagger})) = x_T.
\end{equation*}
Thus, any switching state $x^{\dagger}$ is a root of the function $g(x) = x-x_T + C(\mu_I+f(x))$. Since $f(\cdot)$ is strictly concave, so is $g(\cdot)$, concluding that $g(\cdot)$ can have at most two real roots. Let us denote these roots $x^{\dagger}_1<x^{\dagger}_2$. Then, $x(t^{\dagger})\in\{x^{\dagger}_1,x^{\dagger}_2\}$.

Since $g(\cdot)$ is strictly concave and differentiable, with $g(x^{\dagger}_1)=g(x^{\dagger}_2)=0$, there exists $\tilde x\in (x^{\dagger}_1,x^{\dagger}_2)$ such that $g'(\tilde x)=0$, and then $g'(x^{\dagger}_1)>0$ and $g'(x^{\dagger}_2)<0$. Observe that if the switch is performed at $x(t^{\dagger})=x^{\dagger}_1$, then
\begin{equation*}
g'(x^{\dagger}_1) = g'(x(t^{\dagger})) = 1+f'(x(t^{\dagger})) = \dot \lambda(t^{\dagger})<0,
\end{equation*}
which implies that the switch is from $w=\mu_I$ to $w=0$. Similarly, if the switch is performed at $x(t^{\dagger})=x^{\dagger}_2$, then $\dot \lambda(t^{\dagger})>0$, so the switch is from $w=0$ to $w=\mu_I$.

We claim that, if the switching point is $x(t^{\dagger})=x^{\dagger}_1$, there cannot exist a future switch. Suppose, by contradiction, the existence of another switching time $t^{\ddagger}$, with $t^{\dagger}<t^{\ddagger}<T$, such that for all $t\in(t^{\dagger},t^{\ddagger})$, $0<\lambda(t)<C$ . Thus, necessarily $x(t^{\ddagger})=x^{\dagger}_2$. Since $\lambda(T)=0$, there must exist another switching time $t^{\S}$, with $t^{\ddagger}<t^{\S}<T$, such that $\lambda(t)>C$ for all $t\in(t^{\ddagger},t^{\S})$, and necessarily $x(t^{\S})=x^{\dagger}_1$. Since the ODE system for $(x,\lambda)$ is time-autonomous, there must exist an infinite sequence of switches, with nonvanishing length between switches. Then, the final condition $\lambda(T)=0$ is never satisfied, contradicting that $(x(\cdot),\lambda(\cdot))$ is a Pontryagin extremal.
\end{proof}

\begin{remark}\label{rem:switches}
From Proposition \ref{prop:switches} we observe that for any value $x_T$ of the Hamiltonian ($H=x_T$), the solutions $x^{\dagger}_1$ and $x^{\dagger}_2$ of 
\begin{equation*}
x+C(\mu_I+f(x))-x_T=0
\end{equation*} 
satisfy
\begin{equation*}
\frac{x^{\dagger}_1+x^{\dagger}_2}{2} = \frac{1+BC}{2AC},
\end{equation*}
which is constant respect to $x_T$. So, given an extremal trajectory, its switching states are at most two, and one is a reflection of the other respect to $\frac{1+BC}{2AC}$.
\end{remark}

\subsection{Feasible curves}

We take advantage of the constancy of the Hamiltonian, to generate a simpler characterization of the Pontryagin extremals and, in consequence, optimal solutions. For this, we start by defining the concept of feasible curves in the $(x,\lambda)$ state space.

\begin{definition}\label{def:feasible_curves}
We say that a continuous curve $\zeta(\xi)=(x(\xi),\lambda(\xi))$, $\xi\in[\xi_0,\xi_f]$ in the $(x,\lambda)-$state space is feasible if $\lambda(\xi_f)=0$ and $H(x(\xi),\lambda(\xi);\tilde w(x(\xi),\lambda(\xi)))=x(\xi_f)$ a.e. $\xi\in[\xi_0,\xi_f]$, where $\tilde w(\cdot)$ is as in \eqref{eq:control_min_ham}.
\end{definition}

\begin{remark}\label{rem:lambdas}

Consider an extremal trajectory $(x(\cdot),\lambda(\cdot))$ with final point $(x_T,0)$. The part of the curve that lies in the set $\left\{(x,\lambda)\,|\,\lambda<C\right\}$ satisfies $x +\lambda(\mu_I+f(x)) = x_T$. This allows to define a function $\lambda_{\inf}(x;x_T)$ such that $x\mapsto(x,\lambda_{\inf}(x;x_T))$ is a feasible curve. This function is defined by
\begin{equation}\label{eq:def_lambda_inf}
\lambda_{\inf}(x;x_T) = \frac{x_T-x}{\mu_I+f(x)} .
\end{equation}

Now, consider an extremal trajectory $(x(\cdot),\lambda(\cdot))$, with a part of the trajectory in the set $\left\{(x,\lambda)\,|\,\lambda>C\right\}$, and passing  through the point $(\xref,C)$. Then, it satisfies $(C-\lambda)\tildemuI + x +\lambda(\mu_I+f(x)) = \xref + C(\mu_I+f(\xref))$. This allows to define a function $\lambda_{\sup}(x;\xref)$ such that the Hamiltonian is constant along the curve $x\mapsto(x,\lambda_{\sup}(x;\xref))$. This function is defined by
\begin{equation}\label{eq:def_lambda_sup}
\lambda_{\sup}(x;\xref) = \frac{\xref-x+Cf(\xref)}{f(x)}.
\end{equation}
\end{remark}

Observe that the feasible curves, defined in Definition \ref{def:feasible_curves}, can be constructed from the functions $\lambda_{\inf}(\cdot),\lambda_{\sup}(\cdot)$. Indeed, Remark \ref{rem:lambdas} has all the necessary information for the construction of feasible curves. Our interest in these feasible curves lies in the fact that Pontryagin extremals (and thus, optimal trajectories) are indeed feasible curves.

Let us define $\lambdancl(\cdot)$ as the $\lambda-$nullcline, which, from \eqref{eq:eq_adjunta}, has an explicit expression as a function of $x$:
\begin{equation}\label{eq:lambda_nulclina}
\lambdancl(x):=-1/f'(x)=-1/(B-2Ax), \quad x\neq \frac{B}{2A},
\end{equation} 
and we denote by $\xbar$ the intersection point between $\lambdancl(\cdot)$ and the set $\{(x,\lambda)\,|\,\lambda=C\}$ (where switches occur), that satisfies $Cf'(\xbar)=-1$. This point has the explicit expression
\begin{equation}\label{eq:def_xbar}
\xbar:=\frac{1+BC}{2AC}.
\end{equation}
Notice that this is the point that appears in Remark \ref{rem:switches}. On the other hand, from the expression of the switching function $\psi(t) = \psi(x(t),\lambda(t))$ given in \eqref{eq:sfpsi}, we notice that $\xbar$ defined in \eqref{eq:def_xbar} is the unique state (due to the concavity of $f(\cdot)$)  such that $\psi=\dot \psi =0$, that is, $\xbar$ is a singular arc. Nevertheless, from Proposition \ref{prop:miscelaneos_Pontryagin} we know there are no optimal singular arcs, therefore, according to \cite{BoscainPiccoli} the point $\xbar$ is an anti-turnpike singular arc.

\begin{equation}\label{eq:def_xbar2}
\rbar := \frac{\rmumas+\rzeromas}{2},
\end{equation}
the mean value of the positive equilibria of the dynamics under $w=\mu_I$ (denoted by $\rmumas$) and $w=0$ (denoted by $\rzeromas$),  with $\rmumas,\rzeromas$ as in \eqref{eq:rmui}-\eqref{eq:rzero}. As Figure \ref{fig:diag_fase} shows, the place of $\xbar$ with respect to $\rmumas$, $\rzeromas$, and $\rbar$ determines the qualitative behavior of the feasible curves given by the functions $\lambda_{\inf}(\cdot)$ and $\lambda_{\sup}(\cdot)$.

\begin{center}
\includegraphics[scale=0.33]{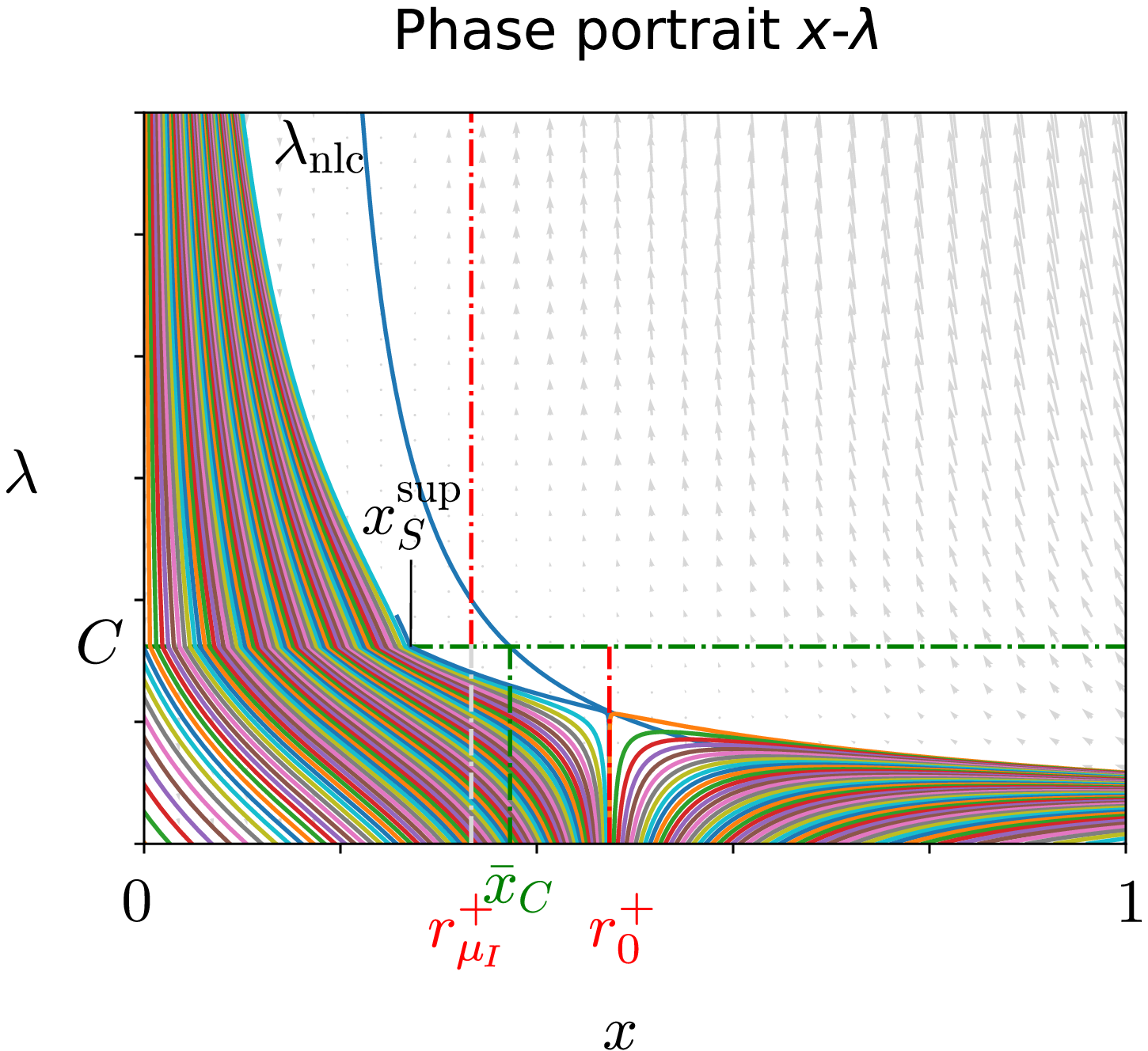}\ \
\includegraphics[scale=0.33]{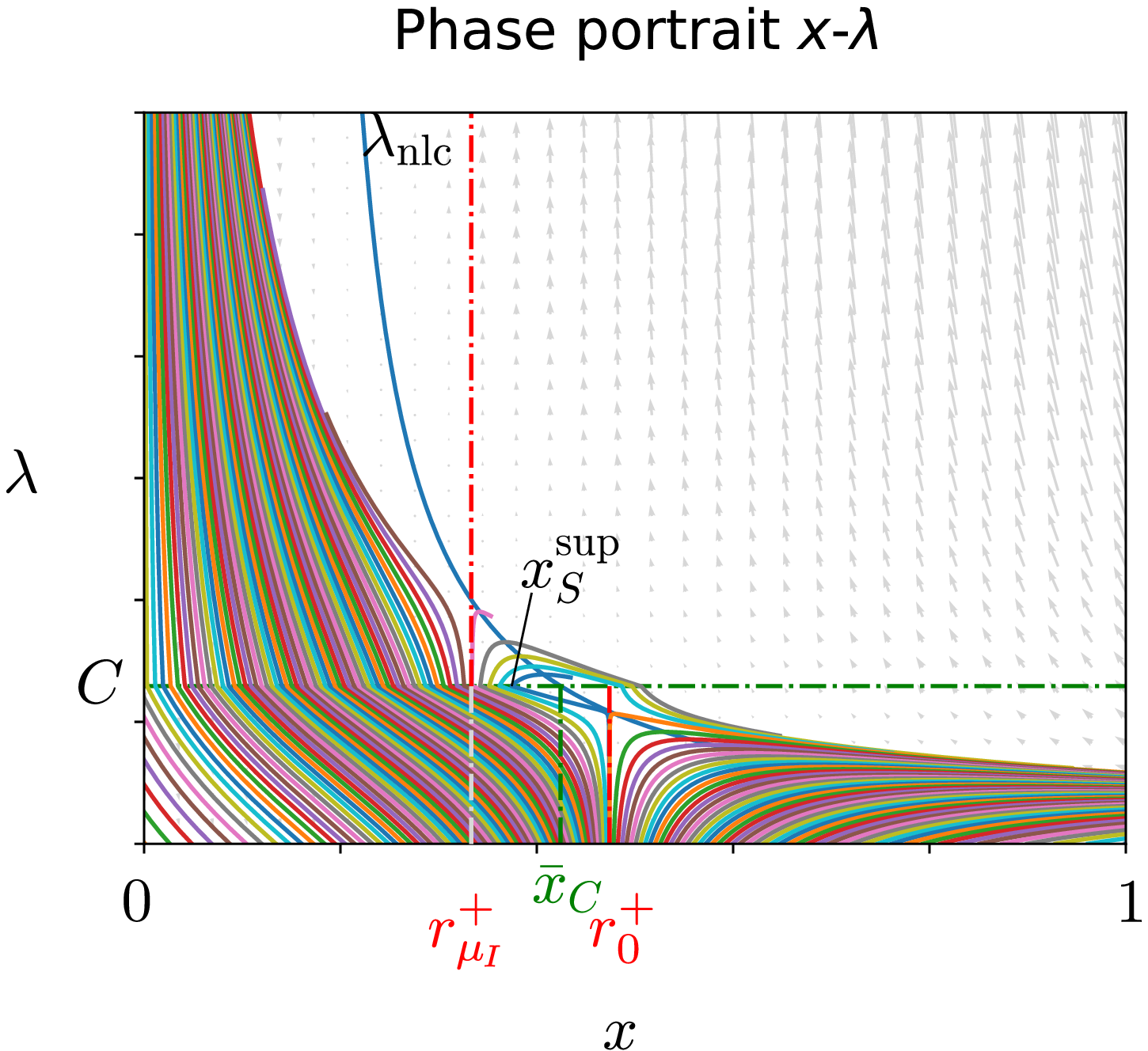}\ \
\includegraphics[scale=0.33]{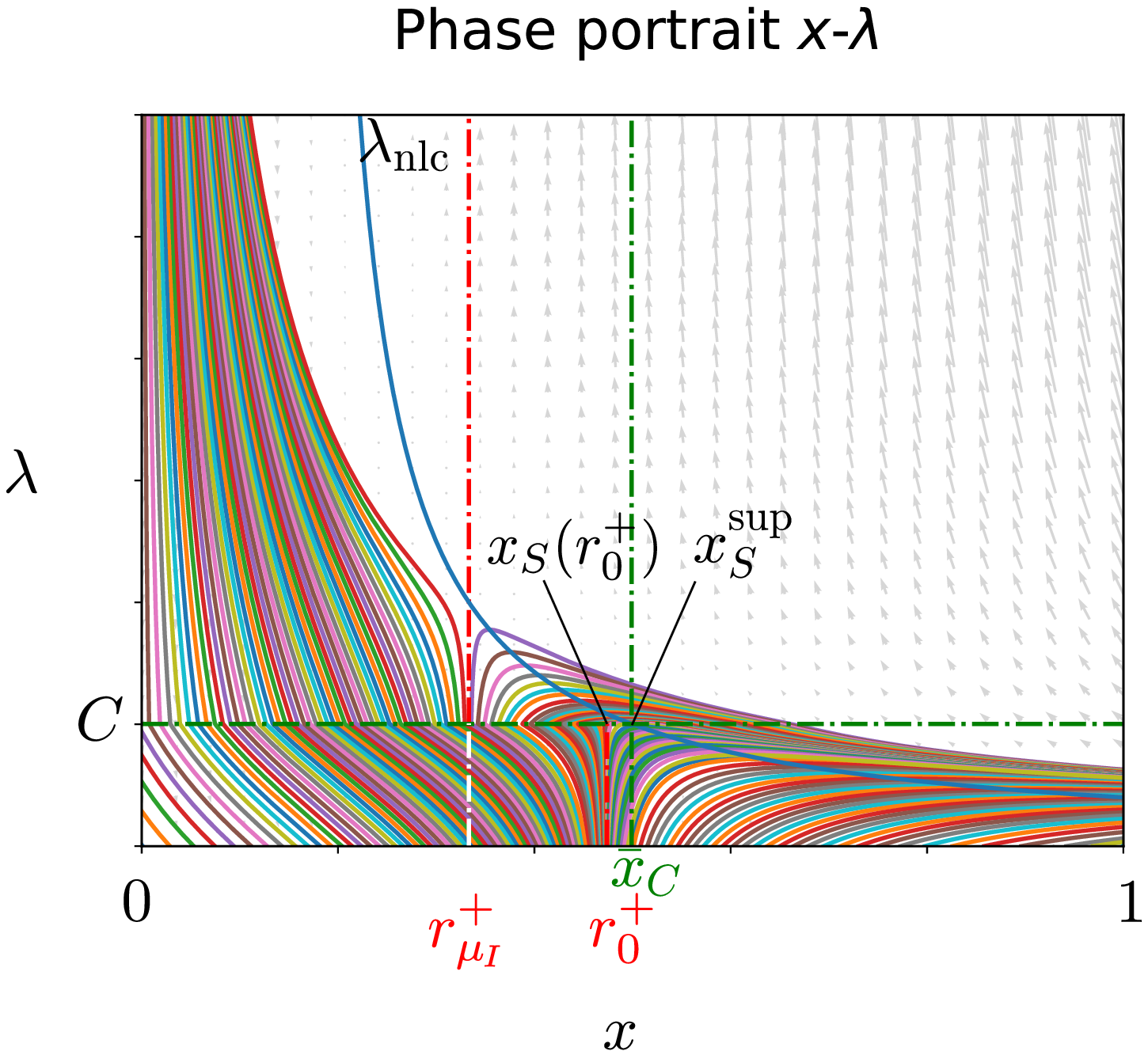}\ \
\captionof{figure}{Diagram of the feasible curves, given a mesh of final conditions $x_T$, in the $(x,\lambda)-$state space: Left, $\xbar<\rbar$. Center, $\rbar<\xbar<\rzeromas$. Right, $\xbar>\rzeromas$.}\label{fig:diag_fase}
\end{center}

\begin{remark}\label{rem:params_y_rs}
Notice that 
\begin{enumerate}\itemsep0em
\item $\xbar<\rmumas$  if and only if $1<BC$; 
\item $\xbar<\rbar$ if and only if $2<C(B+\sqrt{\Delta})$;
\item $\xbar\geq\rzeromas$ if and only if $1\geq C\sqrt{\Delta}$.
\end{enumerate}
\end{remark}

Regarding the function $\lambda_{\inf}(\cdot)$, given by \eqref{eq:def_lambda_inf}, we have the following lemma:

\begin{lemma}\label{lemma:lambda_inf} 
For $\lambda_{\inf}(\cdot)$, we have:
\begin{enumerate}
\item For each fixed $x_T\in(\rzeromenos,\rzeromas)$, the function $x\mapsto\lambda_{\inf}(x;x_T)$ is decreasing in $(\rzeromenos,\rzeromas)$, strictly positive in $(\rzeromenos,x_T)$, has a zero in $x_T$, is negative in $(x_T,\rzeromas)$, and strictly positive in $(\rzeromas,1]$. Moreover, $\lim_{x\searrow\rzeromenos}\lambda_{\inf}(x;x_T)=\infty$ and $\lim_{x\nearrow\rzeromas}\lambda_{\inf}(x;x_T)=-\infty$.

\item For each fixed $x_T\in(\rzeromas,1]$, the function $x\mapsto\lambda_{\inf}(x;x_T)$ is strictly positive in $(\rzeromenos,\rzeromas)$, negative in $(\rzeromas,x_T)$, has a zero in $x_T$, is strictly positive in $(x_T,1]$, and has a unique maximum $x^{\dagger}=x^{\dagger}(x_T)>x_T>\rzeromas$ characterized as the unique solution of $\lambda_{\inf}(x^{\dagger};x_T)=\lambdancl(x^{\dagger})$ (with $\lambdancl(\cdot)$ defined in \eqref{eq:lambda_nulclina}).

\item For each fixed $x<\rzeromas$, the function $x_T\mapsto \lambda_{\inf}(x;x_T)$ is increasing in $x_T\in(\rzeromenos,\rzeromas)$. 

\item For each fixed $x>\rzeromas$, the function $x_T\mapsto \lambda_{\inf}(x;x_T)$ is decreasing in $x_T\in(\rzeromas,1]$. 

\end{enumerate}
\end{lemma}

\begin{proof}
We remark that the zeros of the denominator of $\lambda_{\inf}(\cdot)$ are $\rzeromenos<0$ and $\rzeromas>0$, thus $x\mapsto \mu_I+f(x)$ is positive in $(\rzeromenos,\rzeromas)$ and negative in $(\rzeromas,1]$.
\begin{enumerate}
\item The part of the signs is direct. Now, 
\begin{equation}\label{eq:d_lambda_inf}
\frac{\partial\lambda_{\inf}}{\partial x} (x;x_T) = -\frac{\mu_I+f(x)+(x_T-x)f'(x)}{(\mu_I+f(x))^2}.
\end{equation}
Thus, 
\begin{equation*}\label{eq:xdagger}
\frac{\partial\lambda_{\inf}}{\partial x} (x^{\dagger};x_T)=0\quad \Leftrightarrow\quad x^{\dagger} = x_T\pm \sqrt{ x_T^2 - \frac{\mu_I+Bx_T}{A} },
\end{equation*}
the solutions being real if and only if $\mu_I+f(x_T)\leq0$, which is not possible if $x_T\in(\rzeromenos,\rzeromas)$. Thus, $\frac{\partial}{\partial x} \lambda_{\inf}(\,\cdot\,;x_T)$ does not change its sign, concluding that $\lambda_{\inf}(\,\cdot\,;x_T)$ is decreasing for $x\in(\rzeromenos,\rzeromas)$. Now, the convergences are direct, since the zeros of the denominator of $\lambda_{\inf}(\cdot)$ are $\rzeromas$ and $\rzeromenos$. 

\item The part of the signs is direct. Moreover, $\lim_{x\searrow\rzeromas}\lambda_{\inf}(x;x_T)=-\infty$, and $\lim_{x\nearrow\infty}\lambda_{\inf}(x;x_T)=0$, so the existence of a maximum is guaranteed in the interval $(\rzeromas,\infty)$. Now, if $x^{\dagger}$ is a maximum of $\lambda_{\inf}(\,\cdot\,;x_T)$, it satisfies $\frac{\partial}{\partial x} \lambda_{\inf}(x^{\dagger};x_T)=0$. Using \eqref{eq:d_lambda_inf}, we have
\begin{equation}\label{eq:auxx01}
\frac{\partial\lambda_{\inf}}{\partial x} (x^{\dagger};x_T)=0\quad \Leftrightarrow\quad x_T-x^{\dagger} = -\frac{\mu_I+f(x^{\dagger})}{f'(x^{\dagger})}.
\end{equation}
Thus, replacing the previous expression in \eqref{eq:def_lambda_inf}
\begin{equation}\label{eq:auxx02}
\lambda_{\inf}(x^{\dagger};x_T) = \frac{-1}{f'(x^{\dagger})} = \lambdancl(x^{\dagger}).
\end{equation}
Now, \eqref{eq:auxx01} implies that any local maximizer/minimizer $x^{\dagger}$ of $x\mapsto\lambda_{\inf}(x;x_T)$ satisfies $x_T=g(x^{\dagger})=x^{\dagger}-\frac{\mu_I+f(x^{\dagger})}{f'(x^{\dagger})}$. Notice that $g'(x):=\frac{(\mu_I+f(x))f''(x)}{f'(x)^2}$, which is strictly positive for $x\geq x_T>\rzeromas$. Then, we conclude that \eqref{eq:auxx01} has only one solution $x^{\dagger}>x_T>\rzeromas$; that is, $x\mapsto\lambda_{\inf}(x;x_T)$ has only one maximum $x^{\dagger}(x_T)$ in the set $(\rzeromas,1]$, characterized by \eqref{eq:auxx02}.

\item Since, for any $x<\rzeromas$ fixed, the denominator of $\lambda_{\inf}(x;x_T)$ is positive, $x_T\mapsto\lambda_{\inf}(x;x_T)$ is an increasing affine function.
\item Since, for any $x>\rzeromas$ fixed, the denominator of $\lambda_{\inf}(x;x_T)$ is negative, $x_T\mapsto\lambda_{\inf}(x;x_T)$ is a decreasing affine function.
\end{enumerate}
\end{proof}


For each final state $x_T\in[0,1]$, we denote by $x_S(x_T)$  the last switching point (last with respect to time) of the feasible curve $(x(\cdot),\lambda(\cdot))$ that passes through the point $(x_T,0)$ (final state, final adjoint state) in the $(x,\lambda)$ state space, if such switching point exists. In such case, $x_S(x_T)$ can be found using $\lambda_{\inf}(\cdot)$, as it satisfies 
\begin{equation}\label{eq:def_xs_xt}
\lambda_{\inf}(x_S(x_T);x_T)=C.
\end{equation}

The following lemma characterizes the existence of the state of last switching $x_S(x_T)$, and gives explicit formulas.

\begin{lemma}\label{lemma:xs_xt}
Given a final state $x_T\in[0,1]$, for the last switching point $x_S(x_T)$ we have:
\begin{enumerate}
\item If $x_T<\rzeromas$, then $x_S(x_T)$ always exists; it holds $x_S(x_T)\leq x_T$; and we have the explicit formula 
\begin{equation}\label{eq:xsxtxchico}
x_S(x_T) = \xbar-\sqrt{ \kappa(x_T) }, \mbox{ with } \kappa(x_T):=\xbar^2 + \frac{\mu_I}{A} - \frac{x_T}{CA} .
\end{equation}

\item If $x_T>\rzeromas$, the point $x_S(x_T)$ exists only for the parametric configuration $\rzeromas\leq\xbar$, for $x_T$ in the domain $\rzeromas<x_T\leq\rzeromas+AC(\xbar-\rzeromas)^2$. It holds $x_S(x_T)\geq x_T$; and we have the explicit formula 
\begin{equation}\label{eq:xsxtgrande}
x_S(x_T) = \xbar-\sqrt{ (\xbar-\rzeromas)^2-\frac{x_T-\rzeromas}{AC} }.
\end{equation}

\item If $x_T=\rzeromas$, we can define (with lateral limits, if necessary)
\begin{equation*}
x_S(\rzeromas):= \lim_{x_T\rightarrow\rzeromas}x_S(x_T).
\end{equation*}
In such case, we have $x_S(\rzeromas)=\sup\{x_S(x_T)\,|\, x_T\in(\rzeromenos,\rzeromas)\}$, and it satisfies:
\begin{equation}\label{eq:xs_xhash}
x_S(\rzeromas) = \left\{\quad\begin{split}
2\xbar-\rzeromas,\quad& \mbox{ if }\quad \xbar<\rzeromas,\\
\rzeromas,\quad& \mbox{ if }\quad \rzeromas\leq\xbar .
\end{split}
\right.
\end{equation}
\end{enumerate}
\end{lemma}

\begin{proof}
\begin{enumerate}
\item The switching point $x_S(x_T)$ satisfies $\lambda_{\inf}(x_S(x_T);x_T)=C$, which is equivalent to being a solution of
\begin{equation}\label{eq:eq_xs_xt}
g(x):=-ACx^2+(BC+1)x+C\mu_I-x_T=0.
\end{equation}
The solutions of \eqref{eq:eq_xs_xt} are given by
\begin{equation*}
x_S^{\pm}(x_T) = \xbar\pm\sqrt{ \kappa(x_T)  },~~~\mbox{ with }~~~ \kappa(x_T) = \xbar^2 + \frac{\mu_I}{A} - \frac{x_T}{CA}.
\end{equation*}
Notice that $\kappa(\cdot)$ is decreasing where it is well defined, and then, in order to prove that $x_S^{\pm}(x_T)$ exist for $x_T\leq\rzeromas$, it is enough to prove that $\kappa(\rzeromas)$ is nonnegative. Now, it is not difficult to see that
\begin{equation}\label{eq:ex_xs_xt}
\kappa(\rzeromas) \geq 0 \quad\Leftrightarrow\quad 1+C^2\Delta\geq 2C\sqrt\Delta \quad\Leftrightarrow\quad  (1-C\sqrt{\Delta})^2\geq0,
\end{equation}
which is always true (the equality holds only if $1=C\sqrt{\Delta}$).

Now, since $x\mapsto\lambda_{\inf}(x;x_T)$ is strictly decreasing in the interval $\rzeromenos<x<\rzeromas$ (Lemma \ref{lemma:lambda_inf}), with $\lim_{x\searrow\rzeromenos}\lambda_{\inf}(x;x_T)=\infty$ and $\lim_{x\nearrow\rzeromas}\lambda_{\inf}(x;x_T)=-\infty$, then there exists only one solution $x_S(x_T)\in(\rzeromenos,x_T)$ such that $(x_S(x_T),C)$ and $(x_T,0)$ belong to the same feasible curve. We claim that $x_S(x_T)=x_S^{-}(x_T)$, that is, the solution given in \eqref{eq:xsxtxchico}. Indeed, notice that $g(\cdot)$ defined in \eqref{eq:eq_xs_xt} is a concave quadratic function. Using that $A\rzeromas\,^2=B\rzeromas+\mu_I$, we get that $g(\rzeromas)=\rzeromas-x_T\geq0$. This implies that $x_S^{-}(x_T)\leq\rzeromas< x_S^{+}(x_T)$, from which we conclude that $x_S(x_T)=x_S^{-}(x_T)$.

\item Consider now $x_T>\rzeromas$. In this case we consider $x\mapsto\lambda_{\inf}(x;x_T)$ for $x>\rzeromas$. 
We perform the change of variables $x_T=\rzeromas+y_T$, $x_S(x_T)=\rzeromas + y_S(y_T)$, and $x=\rzeromas+y$ for $x>\rzeromas$. Replacing in \eqref{eq:def_xs_xt}, we obtain that $y_S(y_T)$ satisfies
\begin{equation*}
y_S(y_T)^2-( 2\xbar-\rzeromas )y_S(y_T)+\frac{y_T}{AC} = 0.
\end{equation*}
The two roots of this equation are
\begin{equation*}
y_S^{\pm}(y_T) = (\xbar-\rzeromas)\pm\sqrt{ (\xbar-\rzeromas)^2-\frac{y_T}{AC} },
\end{equation*}
obtaining 
\begin{equation*}
x_S^{\pm}(x_T) = \xbar\pm\sqrt{ (\xbar-\rzeromas)^2-\frac{x_T-\rzeromas}{AC} },
\end{equation*}
this expression being real if and only if $ (\xbar-\rzeromas)^2-\frac{x_T-\rzeromas}{AC} >0$, which is equivalent to $x_T \leq AC(\xbar-\rzeromas)^2 + \rzeromas$. This proves the part of the domain. Let us define by now $x_T^{\sup}=\rzeromas+AC(\xbar-\rzeromas)^2$

We claim that $x_S(x_T)=x_S^{-}(x_T)$ (the solution given in \eqref{eq:xsxtgrande}). Since Lemma \ref{lemma:lambda_inf} states that the function $x\mapsto\lambda_{\inf}(x;x_T)$ is increasing in the interval $(\rzeromas,x^{\dagger}(x_T))$ and decreasing in the interval $(x^{\dagger}(x_T),\infty)$, with a unique maximum $x^{\dagger}(x_T)$, it must hold $\lambda_{\inf}(x^{\dagger}(x_T);x_T)>C$ and $x_S^{-}(x_T)<x^{\dagger}(x_T)<x_S^{+}(x_T)$. Thus, for all $x\in[x_T,x_S^{-}(x_T)]$, the curve $x\mapsto(x,\lambda_{\inf}(x;x_T))$ satisfies $\lambda_{\inf}(x;x_T)\leq C$ and, then, we conclude that the last switch corresponds to $x_S(x_T)=x_S^{-}(x_T)$.\\

Now, consider $x_T\leq\rzeromas+AC(\xbar-\rzeromas)^2$. We have $\lambda_{\inf}(x_T;x_T)=0$ and $\lambda_{\inf}(x^{\dagger}(x_T);x_T)\geq C$. Indeed, it is enough to prove that $\lambdancl(x^{\dagger}(\xtsup))=C$, because in that case, 
\begin{equation*}
\begin{split}
\lambda_{\inf}(x^{\dagger}(x_T);x_T)\geq &\, \lambda_{\inf}(x^{\dagger}(\xtsup);x_T) \\
\geq &\, \lambda_{\inf}(x^{\dagger}(\xtsup);\xtsup) = \lambdancl(x^{\dagger}(\xtsup))=C. 
\end{split}
\end{equation*}

The point $x_T^{\star}$ such that $\lambdancl(x^{\dagger}(x_T^{\star}))=C$ satisfies $x^{\dagger}(x_T^{\star})=\xbar$. Since it also satisfies $\lambda_{\inf}(x^{\dagger}(x_T^{\star}),x_T^{\star})=\lambda_{\inf}(\xbar,x_T^{\star})=C$, we obtain that $x_T^{\star}=\xbar+C(\mu_I+f(\xbar))$. It is not difficult to check that $\xtsup = \xbar + C(\mu_I+f(\xbar))$, concluding that $x_T^{\star}=\xtsup$ and, then, $\lambdancl(x^{\dagger}(\xtsup))=C$.

Now, as $x\mapsto\lambda_{\inf}(x;x_T)$ is strictly increasing for $\rzeromas<x<x^{\dagger}(x_T)$, we conclude $x_S(x_T)\geq x_T$.

\item Notice that, from the point 1. of this proposition, $x_S(x_T)$ always exists if $x_T\in(\rzeromenos,\rzeromas)$. From \eqref{eq:xsxtxchico}, it is easy to see that $\kappa(\cdot)$ is decreasing, and then, $x_S(\cdot)$ is increasing, in $(\rzeromenos,\rzeromas)$. Then, $x_S(\rzeromas)=\sup\{x_S(x_T)\,|\, x_T\in(\rzeromenos,\rzeromas)\}$.

Now, we prove \eqref{eq:xs_xhash}. Suppose that $\xbar<\rzeromas$, which, according to Remark \ref{rem:params_y_rs}, is equivalent to $1<C\sqrt{\Delta}$. From \eqref{eq:ex_xs_xt}, this implies that $\kappa(\rzeromas)>0$ and, then, from \eqref{eq:xsxtxchico},
\begin{equation}\label{eq:aauuxx}
x_S(\rzeromas)=2\xbar-\rzeromas \quad\Leftrightarrow\quad  \rzeromas-\xbar = \sqrt{\kappa(\rzeromas)}.
\end{equation}
Since both terms of the equation at the right-hand side of the equivalence in \eqref{eq:aauuxx} are nonnegative, taking powers in both terms, we get that \eqref{eq:aauuxx} is equivalent to $\left( CB-2AC\xbar + 1 \right)\rzeromas  =  0 $, which is true from the definition of $\xbar$.

Now, suppose that $\rzeromas\leq\xbar$. From point 2. of this proposition, we have 
\begin{equation*}
x_S(x_T) = \xbar-\sqrt{ (\xbar-\rzeromas)^2-\frac{x_T-\rzeromas}{AC} }.
\end{equation*}
Taking limits as $x_T\searrow\rzeromas$, we conclude that $x_S(\rzeromas)=\lim_{x_T\searrow\rzeromas} x_S(x_T)=\rzeromas$.
\end{enumerate}

\end{proof}

Lemma \ref{lemma:xs_xt} allows us to define the domain of the last switching point function $x_S(\cdot)$ where is well defined. Indeed, depending on the parametric configuration, we define
\begin{equation}\label{eq:xtsup}
x_T^{\rm sup} = \left\{\quad \begin{split}
\rzeromas, & \quad \mbox{ if } \xbar<\rzeromas,\\
\rzeromas+AC(\xbar-\rzeromas)^2, & \quad \mbox{ if } \rzeromas\leq\xbar .
\end{split}\right.
\end{equation}
Then, from Lemma \ref{lemma:xs_xt}, the domain of $x_S(\cdot)$ is ${\rm dom\, } x_S = (\rzeromenos,\xtsup]$.

To characterize the optimal trajectories, we need to identify the supremum of the states of last switch. This point, characterized in the following result, gives  information about the initial conditions $x_0$ for which there may be a last switch along their associated optimal trajectories, as we will see in Remark \ref{remark:xssup}.

\begin{corollary}\label{lemma:xs_xhash}
The supremum of last switch states, given by
\begin{equation*}
x_{S}^{\sup} := \sup\{ x_S(x_T)\,\vert\, x_T\in {\rm dom}\, x_S \} ,
\end{equation*}
where $x_S(\cdot)$ is the last switch state function with domain ${\rm dom\, } x_S = (\rzeromenos,\xtsup]$, 
satisfies:
\begin{enumerate}
\item If $\xbar<\rbar$, then $\xssup=x_S(\rzeromas)<\rmumas$.
\item If $\rbar<\xbar<\rzeromas$, then $\rmumas<\xssup=x_S(\rzeromas)<\rzeromas$.
\item If $\rzeromas\leq\xbar$, then $\xssup=\xbar \geq\rzeromas$.
\end{enumerate}
\end{corollary}

\begin{proof}
\begin{enumerate} 
\item Since $\xbar<\rbar<\rzeromas$, then $2\xbar-\rzeromas<\rmumas$. From \eqref{eq:xs_xhash}, and the definition of $x_S^{\sup}$, we conclude $x_S^{\sup}=x_S(\rzeromas)<\rmumas$.
\item Since $\rbar<\xbar<\rzeromas$, then $\rmumas<2\xbar-\rzeromas<\rmumas$. From \eqref{eq:xs_xhash}, and the definition of $x_S^{\sup}$, we conclude $\rmumas<x_S^{\sup}=x_S(\rzeromas)<\rmumas$.
\item From point 2 of Lemma \ref{lemma:xs_xt}, we see that $x_S(\cdot)$ is increasing. So $x_S^{\sup}=x_S(x_T^{\sup})=\xbar\geq\rzeromas$.
\end{enumerate}
\end{proof}

\section{Optimal synthesis and value function}\label{sec:sintesis}

Via Pontryagin techniques, we are able to characterize the number of switches that an optimal strategy can present, depending of the initial state.  Our objective is to provide a full synthesis of the optimal control, with respect to state and time, for all parametric configurations.  In this section, we construct the possible regions of switch in the $(x,t)$ state space, and then, via Hamilton-Jacobi-Bellman techniques, we verify that these regions are indeed regions of switch, and we characterize the optimal controls in feedback form.

\subsection{Switching times analysis}\label{sec:switching}

We aim to characterize the times and states of the switches of the optimal trajectories. From Proposition \ref{prop:switches} in Section \ref{subsec:pontryagin}, we know that every optimal trajectory has at most two switches. In what follows, we refer to the occurrence of a last switch before time $T$, for a given optimal trajectory, as the \emph{last switch} (when it exists). In the case that exists another switch for said trajectory, we refer its occurrence as the \emph{first switch}.

\subsubsection{Last switch}

Each state of last switch $x_s=x_S(x_T)\in[x_S(0),\xssup)$ of an optimal trajectory $x(\cdot)$, where $x_S(\cdot)$ is the last switching point function and $\xssup$ is the supremum of last switch states,  is connected with the final point $x_T=x(T)\in[0,\xtsup)$, where $\xtsup$ defines the domain of the last switching point function $x_S(\cdot)$ by ${\rm dom\, } x_S = (\rzeromenos,\xtsup]$ (see \eqref{eq:xtsup}), via the application of a constant control $w=0$. 
Thus, we can define the inverse function of $x_S(x_T)$, which we will denote by $x_T(x_{s})$, that is characterized as a solution of \eqref{eq:def_xs_xt}, that is
\begin{equation*}
\lambda_{\inf}(x_{s};x_T(x_{s}))=C,
\end{equation*}
where the previous formula makes sense for $x_s\in(\rzeromenos,\xssup)$. From the definition of $\lambda_{\inf}(\cdot)$ in \eqref{eq:def_lambda_inf}, we obtain the explicit formula 
\begin{equation*}\label{eq:xt_xs_body}
x_T(x_{s})=C(\mu_I+f(x_{s}))+x_{s} = -AC(x_{s}-\rzeromenos)(x_{s}-\rzeromas)+x_{s},\quad x_s\in(\rzeromenos,\xssup).
\end{equation*}
\smallskip

We define $t_S(x_s)$ the instant of last switch of an extremal trajectory that passes through the point $(x_s,t_S(x_s))$ in the $(x,t)-$state space, as a function of the state of last switch $x_s$. This function has the explicit formula 
\begin{equation}\label{eq:ts3}
t_S(x_s) = T-\frac{1}{\sqrt{\Delta}}\log\left| \frac{1-AC(x_s-\rzeromas)}{1-AC(x_s-\rzeromenos)} \right| ,
\end{equation}
(see Appendix \ref{app:solutionLogistic}). Regarding the behavior of $t_S(\cdot)$, we have the following property:

\begin{lemma}\label{lemma:t_s}
The function $t_S(\cdot)$ is well defined on $(\rzeromenos,\xssup) \subseteq (-\infty,2\xbar-\rzeromas)$,
is strictly decreasing, concave, and 
\begin{enumerate}
\item If $\xbar\leq\rzeromas$, then $\lim_{x_s\nearrow \xssup}t_S(x_s)=-\infty$;
\item If $\xbar>\rzeromas$, then $\lim_{x_s\nearrow \xssup}t_S(x_s)>-\infty$.
\end{enumerate}
\end{lemma}

\begin{proof}
The argument of $\log|\,\cdot\,|$ in \eqref{eq:ts3} is nonnegative if and only if $x_s\in(-\infty,\frac{1}{AC}+\rzeromenos)\cup(\frac{1}{AC}+\rzeromas,\infty)$. Now, it is not difficult to check that 
\begin{equation*}
2\xbar - \rzeromas = \frac{1}{AC}+\rzeromenos,
\end{equation*}
which proves that $t_S(\cdot)$ is well defined in $(-\infty,2\xbar-\rzeromas)$. Notice that from Corollary \ref{lemma:xs_xhash} and Lemma \ref{lemma:xs_xt} (see \eqref{eq:xs_xhash}), one has $\xssup \le 2\xbar-\rzeromas$, hence the function $t_S(\cdot)$ is well defined for $x_s\in(\rzeromenos,\xssup)$.

Now, the derivative of $t_S(\cdot)$ is 
\begin{equation*}\label{eq:derivada_tS}
t_S'(x_s) = \frac{-AC^2}{(1-AC(x_s-\rzeromas))(1-AC(x_s-\rzeromenos))},
\end{equation*} 
which is negative for $x_s<\frac{1}{AC}+\rzeromenos$. 
%
Now, we notice that $\xssup\leq\frac{1}{AC}+\rzeromenos$. Indeed, from Lemma \ref{lemma:xs_xhash}:
\begin{enumerate}\itemsep0em
\item If $\xbar\leq\rzeromas$, notice that $\xssup=x_S(\rzeromas)=2\xbar-\rzeromas = \frac{1}{AC}+\rzeromenos$. Thus, $t_S(x_s)\searrow-\infty$ if $x_s\nearrow \xssup$.

\item If $\xbar>\rzeromas$, notice that $\xssup=\xbar\leq \frac{1}{AC}+\rzeromenos$. Hence, 
\begin{equation*}
\lim_{x_s\nearrow \xssup}t_S(x_s) = T-\frac{1}{\sqrt{\Delta}}\log\left| \frac{1+C\sqrt{\Delta}}{1-C\sqrt{\Delta}} \right|,
\end{equation*}
which is finite, since $\rzeromas<\xbar$ implies $C\sqrt{\Delta}<1$ (Remark \ref{rem:params_y_rs}).
\end{enumerate}

The second derivative of $t_S(\cdot)$ is
\begin{equation}\label{eq:derivadas_t_s}
t_S''(x_s) = \frac{-2A^2C^3}{(1-AC(x_s-\rzeromas))^2(1-AC(x_s-\rzeromenos))^2}\left( 1-ACx_s+\frac{BC}{2} \right).
\end{equation} 
which is negative if and only if $x_s\in(-\infty,\frac{1}{AC}+\frac{B}{2A})\cap{\rm{dom}}(t_S)$. Since $2\xbar-\rzeromas=\frac{1}{AC}+\rzeromenos<\frac{1}{AC}+\frac{B}{2A}$, then $(-\infty,\frac{1}{AC}+\frac{B}{2A})\cap{\rm{dom}}(t_S)=(-\infty,2\xbar-\rzeromas)$, concluding that $t_S(\cdot)$ is concave on $(-\infty,2\xbar-\rzeromas)$.
\end{proof}

The function $t_S(\cdot)$ is defined on the interval $(-\infty,2\xbar-\rzeromas)$. However, as a switching time, it makes sense only for $x_s\in(\rzeromenos,\xssup)\subseteq(-\infty,2\xbar-\rzeromas)$. Let us define the curves on the $(x,t)-$state space given by:
\begin{equation}\label{eq:set-last-switch}
S:=\{(x_s,t_S(x_s))\,|\,x_s\in(\rzeromenos,\xssup)\},\quad S^{\star}:=\left\{(x_s,t_S(x_s))\,|\,x_s\in\left(\rzeromenos,2\xbar-\rzeromas\right)\right\},
\end{equation}
corresponding to the curve of last switch $S$, and the entire graph of $t_S(\cdot)$. 

Associated to the above curves, let us define the sets lying below these curves, given by
\begin{equation}\label{eq:Theta}
\begin{split}
\Theta :=&\, \left\{ (x,t)\in(\rzeromenos,\xssup)\times[0,T]\,|\, t<t_S(x) \right\},\\ 
\Theta^{\star} :=&\, \left\{ (x,t)\in\left(\rzeromenos,2\xbar-\rzeromas\right)\times[0,T]\,|\, t<t_S(x) \right\} .
\end{split}
\end{equation}

Notice that $S=S^{\star}$ and $\Theta=\Theta^{\star}$ if $\xbar<\rzeromas$, as in such case we have $\xssup=2\xbar-\rzeromas$.

\subsubsection{First switch}

When there exists a second switch, we denote by $x_{\sigma}$ the state of first switch. From Remark \ref{rem:switches}, we know that the state of first switch $x_{\sigma}$ and the state of last switch $x_s$ satisfy
\begin{equation*}
x_{\sigma} = 2\xbar - x_s.
\end{equation*}

Let us define the time of first switch $t_{\sigma}(x_{\sigma})$ (depending on the state of first switch) 
\begin{equation}\label{eq:t_sigma}
t_{\sigma}(x_{\sigma})  =  t_S(2\xbar-x_{\sigma})-\frac{1}{B}\log\left| \frac{(2\xbar-x_{\sigma})\left(x_{\sigma}-\rmumas\right)}{\left(2\xbar-x_{\sigma}-\rmumas\right)x_{\sigma}} \right|,
\end{equation}
for the possible switching points $x_{\sigma}\in(2\xbar-\xssup,2\xbar-\rmumas)$, since  trajectories whose point of last switch belong to $(\rzeromenos,\rmumas)$ cannot have another switch (see Appendix \ref{app:solutionLogistic}). As it was done for $t_S(\cdot)$, from $t_{\sigma}(\cdot)$ we define the curve of first switch $\sigma$ by:
\begin{equation}\label{eq:Sigma}
\sigma:=\{(x_{\sigma},t_{\sigma}(x_{\sigma}))\,|\,x_{\sigma}\in(2\xbar-\xssup,2\xbar-\rmumas)\},
\end{equation}
which, if $\rzeromas<\xbar$, can be extended to the point $(\xbar,t_S(\xbar))$, since in that case $\xssup=\xbar$ (Lemma \ref{lemma:xs_xhash}) and $t_S(\cdot)$ is well defined at $\xbar$ (Lemma \ref{lemma:t_s}). Let us define $\Xi$ as the set on the  $(x,t)-$state space that lies below the curve $\sigma$, that is,
\begin{equation}\label{eq:Xi}
\Xi := \left\{ (x,t)\in(2\xbar-\xssup,2\xbar-\rmumas)\times[0,T]\,|\, t<t_{\sigma}(x) \right\},
\end{equation}
($\Xi$ can also be extended in the case $\rzeromas<\xbar$).\\

\begin{remark}
According to Lemma \ref{lemma:xs_xhash}, the left point of the interval of definition of $t_{\sigma}(\cdot)$, given by $2\xbar-\xssup$, depends of the parametric configuration:
\begin{itemize}
\item $\rbar<\xbar<\rzeromas$ implies $2\xbar-\xssup=2\xbar-x_S(\rzeromas)=2\xbar-(2\xbar-\rzeromas)=\rzeromas$; 
\item $\rzeromas\leq\xbar$ implies $2\xbar-\xssup=2\xbar-\xbar=\xbar$.
\end{itemize} 

Notice that $t_{\sigma}(\cdot)$ has asymptotes at $\rzeromas$ and $2\xbar-\rmumas$. Indeed, the denominator of the argument of $\log|\cdot|$ in \eqref{eq:t_sigma} is null at $x_{\sigma}\in\{0,2\xbar-\rmumas\}$ (where $0$ is outside of the domain of interest) and, since $t_S(\cdot)$ has a vertical asymptote at $x_s=2\xbar-\rzeromas$, then the function $x_{\sigma}\mapsto t_S(2\xbar-x_{\sigma})$ has a vertical asymptote at $x_{\sigma}=\rzeromas$. 
\end{remark}		

The curves of first and last switch intersect if and only if $\rzeromas<\xbar$, and the intersection point is $x=\xssup=\xbar$. In such case, we remark the existence of an extremal trajectory $x\mapsto(x,t_{\Gamma}(x))$ in the $(x,t)-$state space  (parametrized by $x$ instead of $t$), generated by the constant control $w=0$, that passes through the intersection point $(\xbar,t_S(\xbar))$. This curve corresponds to the feasible curve $(x^{\Gamma}(\cdot),\lambda^{\Gamma}(\cdot))$, in the $(x,\lambda)-$state space, that passes through the point $(\xssup,C)=(\xbar,C)$, that is, it satisfies $\lambda_{\inf}(\xssup;x_T(\xssup))=C$. Since $x^{\Gamma}(\cdot)$ is solution of \eqref{eq:SIS-logistic-2}, it satisfies $x^{\Gamma}(t)>\rzeromas$ for all $t\geq0$, and $x^{\Gamma}(t)\rightarrow\rzeromas$ as $t\rightarrow\infty$. This function $t_{\Gamma}(\cdot)$ can be obtained replacing $(x_0,t_0)=(\xbar,t_S(\xbar))$ in \eqref{eq:tiempo_logistica}, obtaining
\begin{equation*}\label{eq:t_kink}
t_{\Gamma}(x)= t_S(\xbar) +\frac{1}{\sqrt{\Delta}}\log\left| \frac{(x-r_0^{-})(\xbar-r_0^{+}) }{(x-r_0^{+})(\xbar-r_0^{-})} \right|.
\end{equation*}

Let us define, in the $(x,t)-$state space, the above described curve $\Gamma$, that will be part of a limit region of the optimal synthesis
\begin{equation*}
\Gamma:=\{(x,t_{\Gamma}(x))\,|\,x\in(\rzeromas,1]\},
\end{equation*}
and the set $\Upsilon$ as the set on the $(x,t)-$state space that lies below the curve $\Gamma$, that is,
\begin{equation}\label{eq:Upsilon}
\Upsilon := \{(x,t)\in (\rzeromas,1]\times[0,T] \,|\, t<t_{\Gamma}(x)\}.
\end{equation}

Regarding the set $S$ and the curve $\Gamma$, we have the following result:

\begin{lemma}
Suppose $\rzeromas<\xbar$. Then, there is only one intersection point between the set of last switch  $S$ (in the $(x,t)$ space), given by \eqref{eq:set-last-switch}, and $\Gamma$, at $x=\xssup=\xbar$. Moreover, $S$ and $\Gamma$ intersect tangentially.
\end{lemma}

\begin{proof}
The derivatives of $t_{\gamma}(\cdot)$ are
\begin{equation}\label{eq:derivadas_t_k}
t_{\Gamma}'(x) = \frac{-1}{A(x-\rzeromas)(x-\rzeromenos)}\,,\quad
t_{\Gamma}''(x) = \frac{2}{A(x-\rzeromas)^2(x-\rzeromenos)^2}\left(x-\frac{B}{2A}\right).
\end{equation} 

Since we are considering $x\in[\rzeromas,\xbar]$, we have $t_{\Gamma}''(x)>0$ and $t_S''(x)<0$ (from \eqref{eq:derivadas_t_s}). Thus, $t_{\Gamma}(\cdot)$ is convex and $t_S(\cdot)$ concave. Now, replacing $x=\xbar=\frac{1+BC}{2AC}$ in \eqref{eq:derivadas_t_k} and \eqref{eq:derivadas_t_s}, we get $t_S'(\xbar) = \frac{-1}{A(\xbar-\rzeromas)(\xbar -\rzeromenos)} = t_{\Gamma}'(\xbar)$. We conclude that the intersection between $S$ and $\Gamma$ is unique and tangential, at the point $(\xbar,t_S(\xbar))$.
\end{proof}

In Figure \ref{fig:diag_fase_xt} we illustrate the behavior of the curves $S,\sigma,\Gamma$, as well as the regions $\Theta,\Xi,\Upsilon$, in the $(x,t)-$state space, for the different parametric configurations.\\

\begin{center}
\includegraphics[scale=0.33]{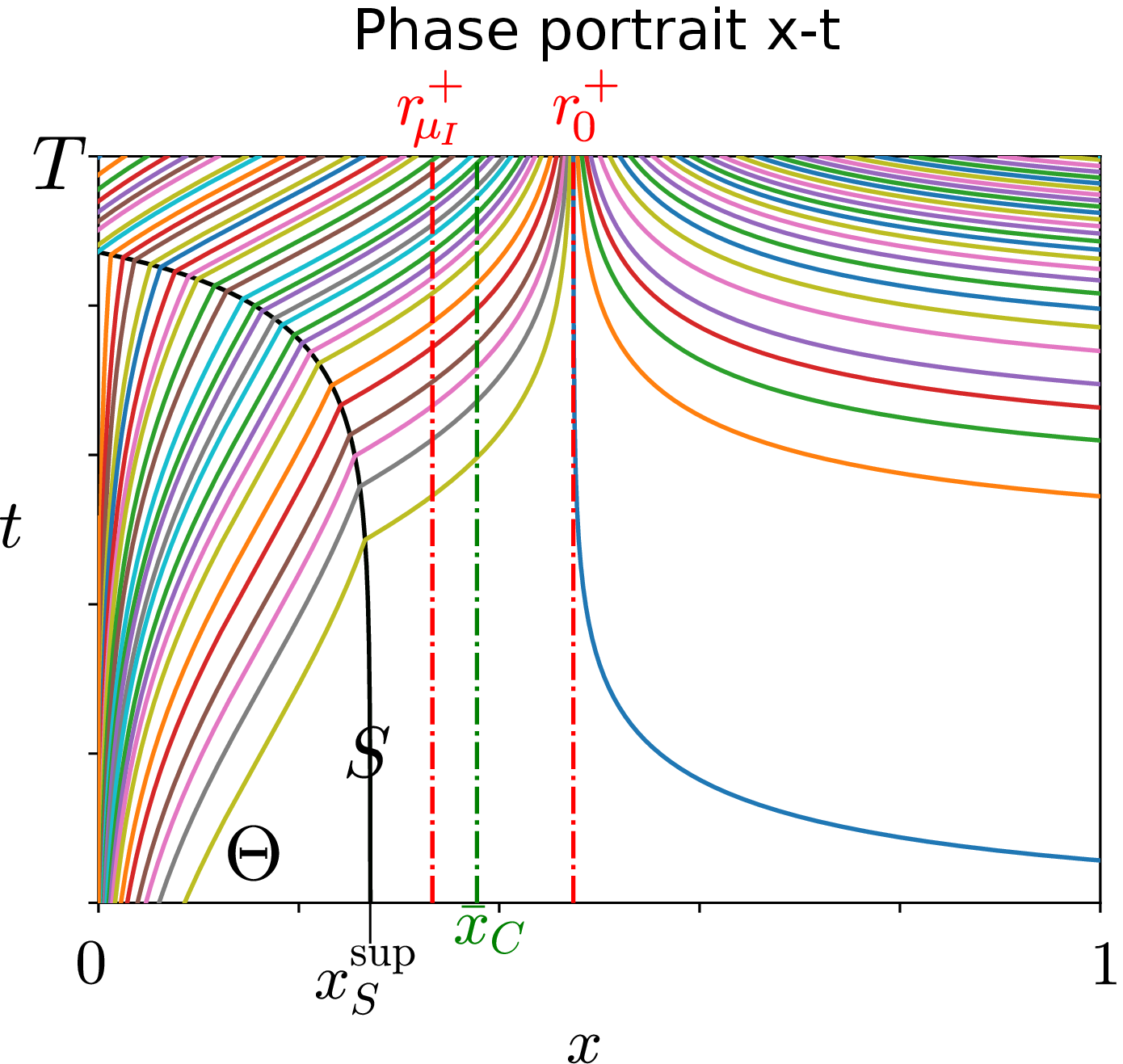}\ \
\includegraphics[scale=0.33]{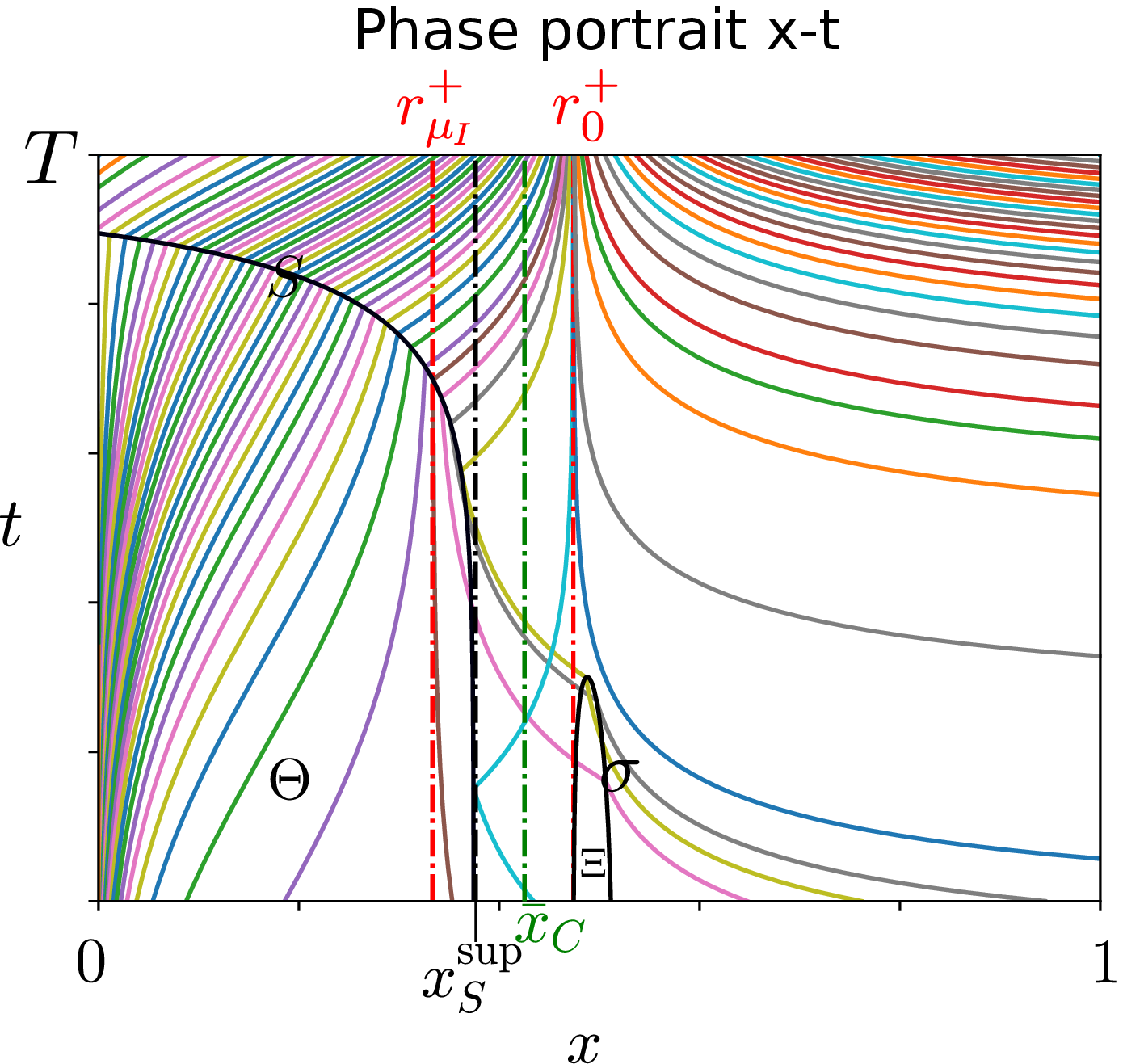}\ \
\includegraphics[scale=0.33]{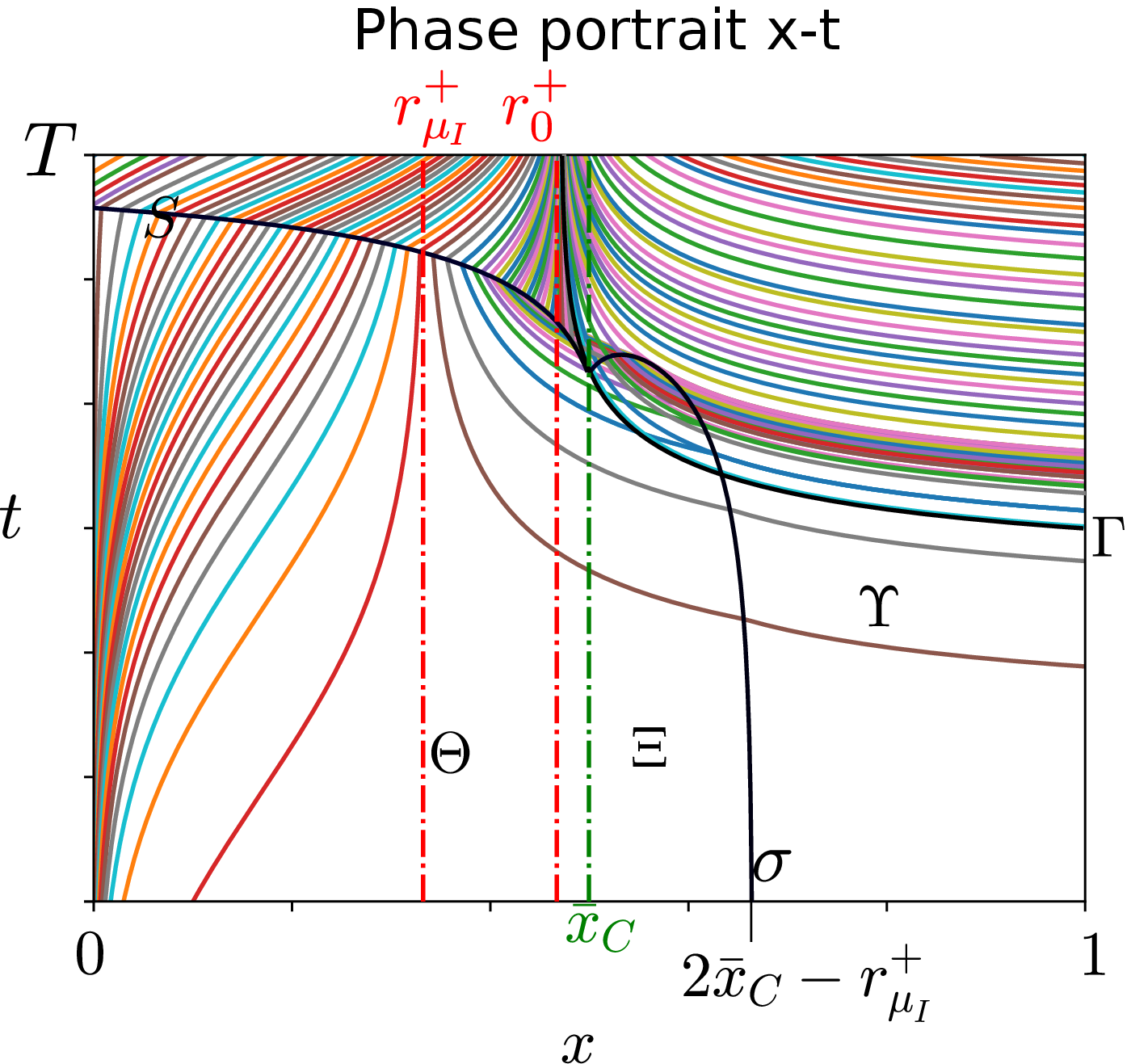}\ \
\captionof{figure}{Diagram of the switching curves and feasible curves, given a mesh of final conditions $x_T$, in the $(x,t)-$state space. Left, $\xbar<\rbar$. Center, $\rbar<\xbar<\rzeromas$. Right, $\xbar>\rzeromas$}\label{fig:diag_fase_xt}
\end{center}

\subsection{Optimal synthesis}

In this section, we characterize the optimal synthesis of the problem with respect to the initial condition and initial time, via Hamilton-Jacobi-Bellman techniques. For this, we recall a key result, found in \cite{BressanPiccoli2007}, which we adapt to the current problem in Lagrange formulation:
\begin{theorem}[\cite{BressanPiccoli2007}, Corollary 7.3.4]\label{teo:bresanpiccoli}
Consider the problem 
\begin{equation*}\label{eq:problem_bressan_piccoli}
V(y,s) = \min \left\{ \int_s^T L(x(t),w(t),t)dt\,\left|\, \substack{\displaystyle\dot x(t) = f(x(t),w(t),t) \mbox{ a.e. } t\in[s,T]\, \\[2mm] \displaystyle x(s)=y  (x(T),T)\in \mathcal T  \\[2mm] \displaystyle w:[s,T]\rightarrow U \mbox{ measurable}} \right. \right\} ,
\end{equation*}
with $U\subseteq\R^m$ compact, $f,L:\R^n\times\R^m\times\R$ continuous in all variables, and continuously differentiable with respect to $x$, and $\mathcal T=\R^n\times \{T\}$ (closed). Let $\mathcal Q=\R^n\times[0,T]$ and let $W:\bar{\mathcal Q}\rightarrow\R$ be a continuous function such that
\begin{enumerate}
\item $W(\cdot)\geq V(\cdot)$,
\item $W(\cdot)=0$ on $\mathcal T=\R^n\times\{T\}$, 
\item There exist finitely many manifolds $\mathcal M_1,\dots,\mathcal M_N\subseteq \mathcal Q$ with dimension $\leq n$ such that $W$ is continuously differentiable and satisfies the Hamilton-Jacobi equation
\begin{equation*}
\frac{\partial W}{\partial s}(y,s) + \min_{\omega\in U} \left\{ \frac{\partial W}{\partial y}(y,s) \cdot f(y,\omega,t) + L(y,\omega,s) \right\} = 0
\end{equation*}
at every point $(y,s)$ in the open set $\mathcal Q\backslash(\cup_{i=1..N} \mathcal M_i)$.
\end{enumerate}
Then, $W(\cdot)$ coincides with the value function $V(\cdot)$ on the closure of the domain $\mathcal Q$.
\end{theorem}

The Hamilton-Jacobi-Bellman equation associated to problem \eqref{eq:problem} is 
\begin{equation}\label{eq:HJB}
\begin{split}
\frac{\partial V}{\partial t_0} + x_0 + (\mu_I+f(x_0))\frac{\partial V}{\partial x_0} + \inf_{w\in[0,\tildemuI]}\left\{ \left(C- \frac{\partial V}{\partial x_0}\right)w \right\}  =&\, 0,\,\\[2mm]
\mbox{for }~~ (x_0,t_0)\in~&[0,1]\times[0,T),\\[2mm]
V(x_0,T) =&\, 0,\, x_0\in[0,1].
\end{split}
\end{equation}

In what follows, we construct a candidate to value function $W(\cdot)$, which will depend on the parameter configuration, and we prove that $W(\cdot)$ satisfies \eqref{eq:HJB}, concluding that $V(\cdot)=W(\cdot)$. From Section \ref{sec:pontryagin}, we know that the optimal controls are bang-bang type, with at most 2 switches. Moreover, we know that the last switch is from $w=\mu_I$ to $w=0$, and when the first switch exists, is from $w=0$ to $w=\mu_I$. Then, we consider an auxiliar cost function, with constant control $w\in[0,\mu_I]$, in a time interval $[t_i,t_f]$, for a trajectory starting at the initial condition $x_i$ at initial time $t_i$: 
\begin{equation*}\label{eq:Jw}
J^{w}(x_i,t_i,t_f) \,:=\, Cw(t_f-t_i) + \int_{t_i}^{t_f}x^{w}(t;x_i,t_i)dt,
\end{equation*}
where $x^{w}(\cdot)$ is the solution of \eqref{eq:SIS-logistic-3} with constant control $w(t)=w\in[0,\mu_I]$. Direct integration of $x^{w}(\cdot)$ leads to the formula
\begin{equation}\label{eq:Jw_expl}
J^{w}(x_i,t_i,t_f) = (r_w^{+}+Cw)(t_f-t_i) + \frac{1}{A}\log\left( \frac{ (x_i-r_w^{-}) - (x_i-r_w^{+})e^{-A(r_w^{+}-r_w^{-})(t_f-t_i)} }{r_w^{+}-r_w^{-}} \right),
\end{equation}
with $\rwmenos,\rwmas$ the steady states of the corresponding dynamics, as in \eqref{eq:raices_w} (see Appendix \ref{app:valueFunction} for details). The partial derivatives of $J^{w}(\cdot)$ are
\begin{equation}\label{eq:partialJ_cuerpo}
\left\{\quad\begin{split}
\frac{\partial J^w}{\partial x_i}(x_i,t_i,t_f) =&\, \frac{1}{A}\frac{1-e^{-A(\rwmas-\rwmenos)(t_f-t_i)}}{(x_i-\rwmenos)-(x_i-\rwmas)e^{-A(\rwmas-\rwmenos)(t_f-t_i)}},\\[2mm]
\frac{\partial J^w}{\partial t_i}(x_i,t_i,t_f) =&\, -(\rwmas+Cw)-\frac{1}{A}\frac{(x_i-\rwmas)e^{-A(\rwmas-\rwmenos)(t_f-t_i)}A(\rwmas-\rwmenos)}{(x_i-\rwmenos)-(x_i-\rwmas)e^{-A(\rwmas-\rwmenos)(t_f-t_i)}},\\[2mm]
\frac{\partial J^w}{\partial t_f}(x_i,t_i,t_f) =&\, -\frac{\partial J^w}{\partial t_i}(x_i,t_i,t_f).
\end{split}\right.
\end{equation}

It is not difficult to verify that, for any fixed final time $t_f$, $J^w(\cdot)$ satisfies the Hamilton-Jacobi equation
\begin{equation*}\label{eq:HJ}
\begin{split}
\frac{\partial J^w}{\partial t_i}(x_i,t_i,t_f) + (Cw + x_i) + (\mu_I-w+f(x_i))\frac{\partial J^w}{\partial x_i}(x_i,t_i,t_f) =&\, 0,\, ,\\[2mm] \mbox{for }~~(x_i,t_i)\in&~[0,1]\times(0,t_f),\\[2mm]
J^w(x_i,t_f,t_f) =&\, 0,\, x_i\in[0,1].
\end{split}
\end{equation*}

In particular, for the case $w=0$, if $t_i<t_f$
\begin{equation}\label{eq:HJ_0}
\frac{\partial J^0}{\partial t_i}(x_i,t_i,t_f) + x_i + (\mu_I+f(x_i))\frac{\partial J^0}{\partial x_i}(x_i,t_i,t_f) \,=\, 0,
\end{equation}
and when $w=\tildemuI$, 
\begin{equation*}\label{eq:HJ_mu}
\frac{\partial J^{\tildemuI}}{\partial t_i}(x_i,t_i,t_f) + (C\tildemuI + x_i) + f(x_i)\frac{\partial J^{\tildemuI}}{\partial x_i}(x_i,t_i,t_f) \,=\, 0.
\end{equation*}
\medskip

The value function must be a combination of the functionals $J^0(\cdot)$ and $J^{\tildemuI}(\cdot)$. In Section \ref{sec:switching} we have identified the regions in the $(x,t)-$state space in which a switch can occur, which are the curves $S$ (last switch, given by \eqref{eq:set-last-switch}) and $\sigma$ (first switch, given by \eqref{eq:Sigma}). Let us define 
\begin{equation}\label{eq:hat_ts_xs}
\hat t_S(x_0,t_0) = \inf\{ t>t_0\,|\, (x^{\mu_I}(t;x_0,t_0),t)\in S \},\quad
\hat x_S(x_0,t_0) = x^{\mu_I}(\hat t_S(x_0,t_0);x_0,t_0),
\end{equation}
the first time that a trajectory controlled by $w=\mu_I$, starting from the point $(x_0,t_0)$, intersects the set of last switch (on the $(x,t)-$state space) $S$, given by \eqref{eq:set-last-switch}, and $\hat x_S(x_0,t_0)$ is the corresponding switching state. Similarly, if the curve of first switch $\sigma$ exists, given by \eqref{eq:Sigma}, we define
\begin{equation}\label{eq:hat_tsigma_xsigma}
\hat t_{\sigma}(x_0,t_0) = \inf\{ t>t_0\,|\, (x^{0}(t;x_0,t_0),t)\in \sigma \}, \quad 
\hat x_{\sigma}(x_0,t_0) = x^{0}(\hat t_{\sigma}(x_0,t_0);x_0,t_0),
\end{equation}
the first time that a trajectory controlled by $w=0$, starting from the point $(x_0,t_0)$, intersects the set $\sigma$ (on the $(x,t)-$state space),  and $\hat x_{\sigma}(x_0,t_0) $ is the corresponding switching  state. 

Thus, a point $(\hat x_S(x_0,t_0),\hat t_S(x_0,t_0))$ is likely to be a switching point for the dynamics, from $w=\mu_I$ to $w=0$. Similarly, a point $(\hat x_{\sigma}(x_0,t_0),\hat t_{\sigma}(x_0,t_0))$ is likely to be a switching point for the dynamics, from $w=0$ to $w=\mu_I$.

For initial conditions $(x_0,t_0)$ such that $\hat t_S(x_0,t_0)<\infty$, consider the cost:
\begin{equation}\label{eq:WS}
W_S(x_0,t_0) = J^{\mu_I}(x_0,t_0,\hat t_S(x_0,t_0)) + J^{0}(\hat x_S(x_0,t_0),\hat t_S(x_0,t_0),T).
\end{equation}

Now, we present the two main results of the section, distinguishing the two cases $\xbar\leq\rzeromas$ and $\rzeromas<\xbar$, for which the switching curves have different behaviors. These results are established in propositions  \ref{prop:HJB_sintesis_1} and \ref{prop:HJB_sintesis_2} below. Before presenting their proofs, we introduce three technical lemmas.

\begin{proposition}\label{prop:HJB_sintesis_1}
Suppose $\xbar\leq\rzeromas$. Let us define the feedback control
\begin{equation}\label{eq:control_caso1}
\hat w(x,t) := 
\left\{\,\begin{split}
\,\tildemuI,&\quad \mbox{ if } (x,t)\in\Theta,\\ 
\,0,&\quad \mbox{ if } (x,t)\notin\Theta,\\
\end{split}\right.
\end{equation}
and its associated cost function
\begin{equation}\label{valueFunction1}
W(x_0,t_0) = \left\{\quad\begin{split}
J^{0}(x_0,t_0,T), & \,\quad (x_0,t_0)\notin\Theta,\\
W_S(x_0,t_0), & \,\quad (x_0,t_0)\in\Theta ,
\end{split}\right.
\end{equation}
where $W_S(x_0,t_0)$ is defined in \eqref{eq:WS}. Then, the function $W(\cdot)$ defined in \eqref{valueFunction1} is a viscosity solution of the Hamilton-Jacobi-Bellman equation \eqref{eq:HJB}, and therefore  $\hat w$ given by \eqref{eq:control_caso1} is an optimal feedback control.
\end{proposition}

\begin{proposition}\label{prop:HJB_sintesis_2}
Suppose $\xbar>\rzeromas$. Let us define $\mathcal{V}:={\rm epi}(t_S(\cdot)|_{[0,\xbar]}) \cup {\rm epi}(t_{\Gamma}(\cdot)|_{[\xbar,1]})$ (where ${\rm epi}(\cdot)$ stands for the epigraph of a function), $\mathcal{S}:=\Theta\cup(\Xi\cap\Upsilon)$, $\mathcal{T}=([0,1]\times[0,T])\backslash(\mathcal{V}\cup\mathcal{S})$, with $\Theta,\Xi,\Upsilon$ as in \eqref{eq:Theta},\eqref{eq:Xi},\eqref{eq:Upsilon}, the feedback control
\begin{equation}\label{eq:control_caso2}
\hat w(x,t) := 
\left\{\,\begin{split}
\,\tildemuI,&\quad \mbox{ if } (x,t)\in\mathcal{S},\\ 
\,0,&\quad \mbox{ if } (x,t)\notin\mathcal{S},\\
\end{split}\right.
\end{equation}
and its associated cost function
\begin{equation}\label{valueFunction2}
W(x_0,t_0) = \left\{\,\begin{split}
J^{0}(x_0,t_0,T), & \,\quad (x_0,t_0)\in\mathcal{V},\\
W_S(x_0,t_0), & \,\quad (x_0,t_0)\in\mathcal{S}, \\
J^{0}(x_0,t_0,\hat t_{\sigma}(x_0,t_0)) + W_S(\hat x_{\sigma}(x_0,t_0),\hat t_{\sigma}(x_0,t_0)), & \,\quad (x_0,t_0)\in\mathcal{T}, \\
\end{split}\right.
\end{equation}
with $(\hat t_{S}(\cdot),\hat x_{S}(\cdot))$ and $(\hat t_{\sigma}(\cdot),\hat x_{\sigma}(\cdot))$ as in \eqref{eq:hat_ts_xs},\eqref{eq:hat_tsigma_xsigma}. Then, the function $W(\cdot)$ defined in \eqref{valueFunction2} is a viscosity solution of the Hamilton-Jacobi-Bellman equation \eqref{eq:HJB}, and therefore  $\hat w$ given by  \eqref{eq:control_caso2} is an optimal feedback control.
\end{proposition}
\medskip

An illustration of the optimal synthesis is shown in Figure \ref{fig:systhesis}.

\begin{center}
\includegraphics[scale=0.4]{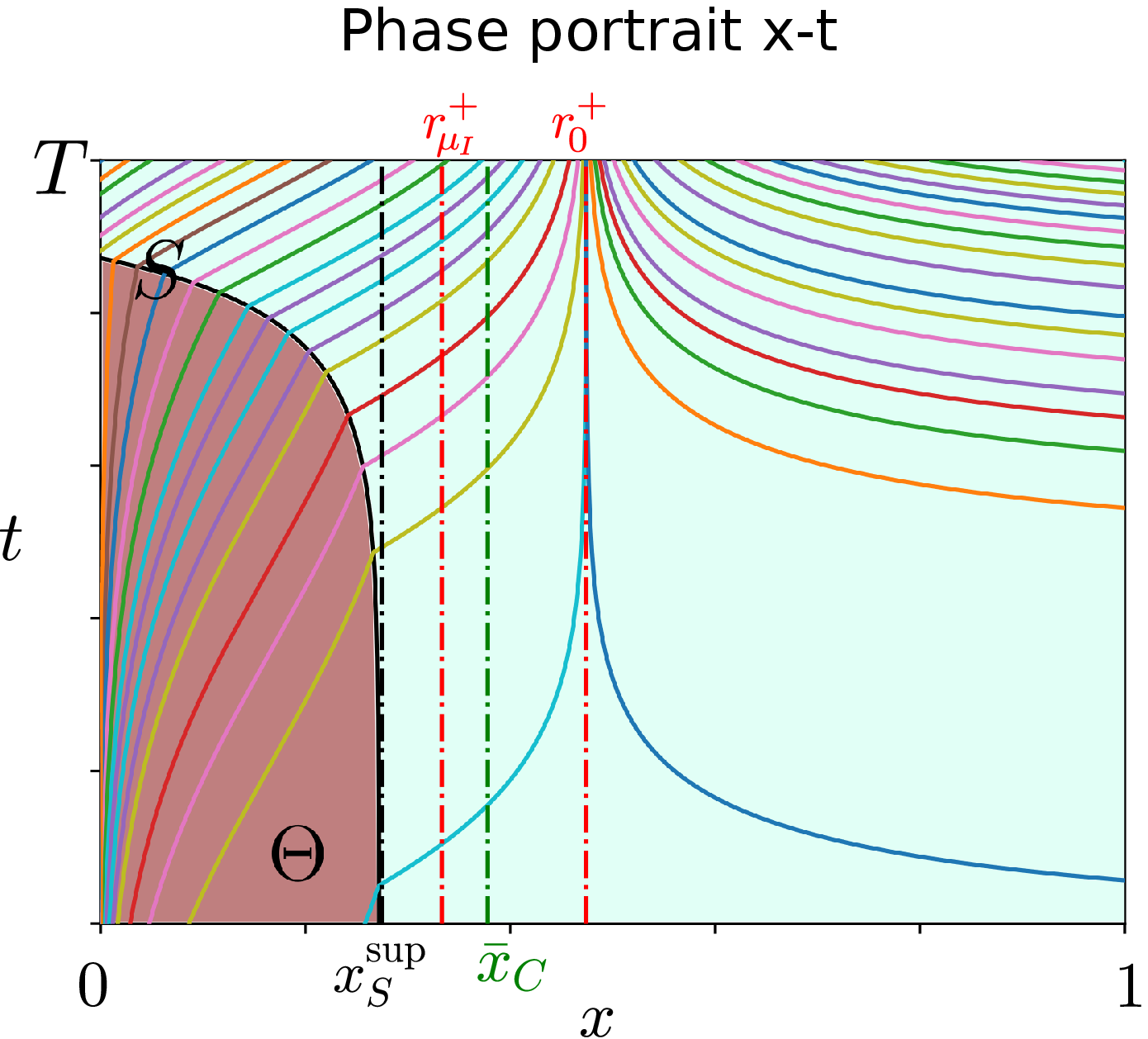}
\ \
\includegraphics[scale=0.4]{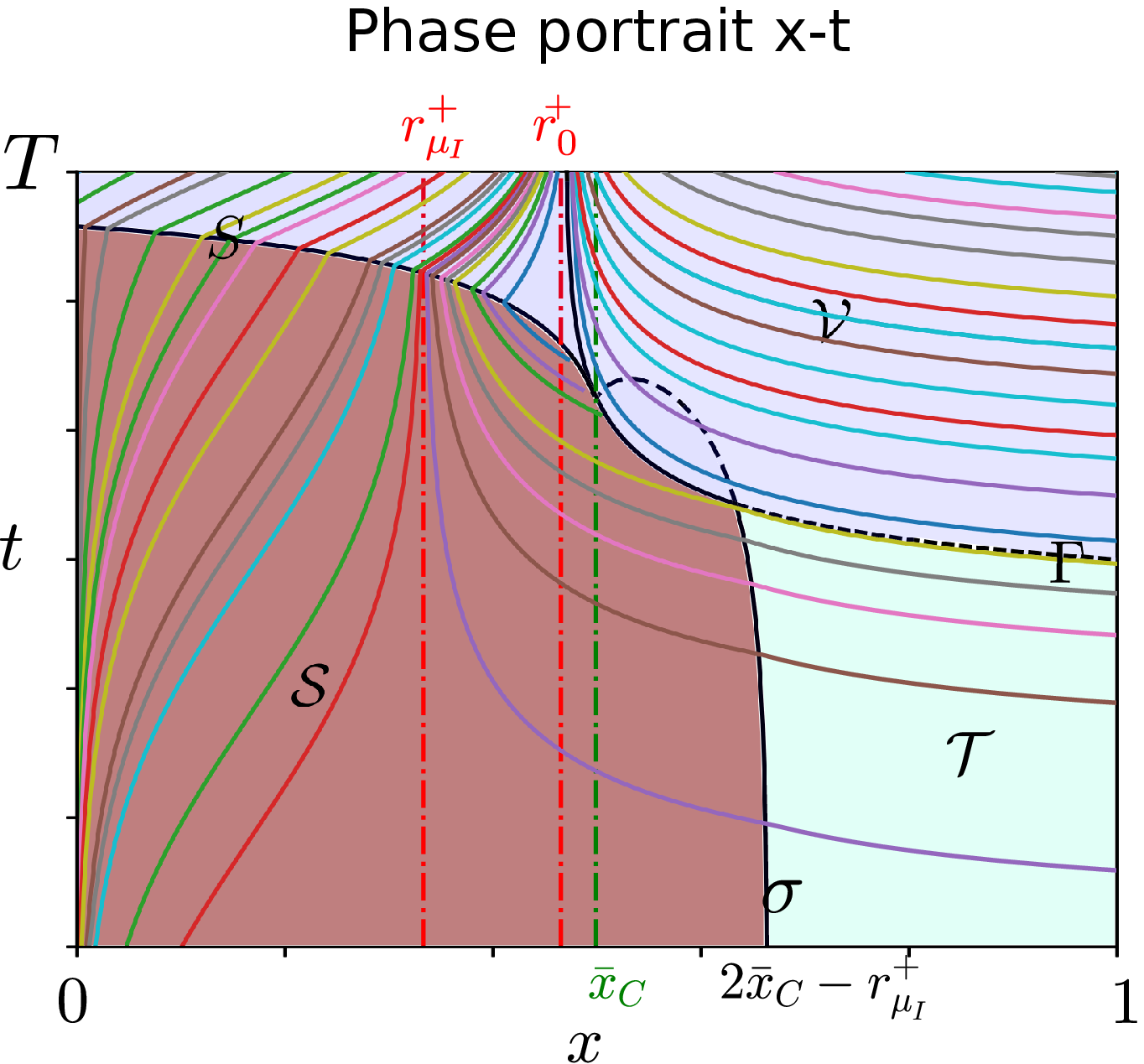}
\captionof{figure}{Optimal synthesis. Left, the case $\xbar\leq\rzeromas$. Right, $\xbar>\rzeromas$. In the dark red regions, the optimal control value is $w=\tildemuI$. In the rest, $w=0$.}\label{fig:systhesis}
\end{center}

To prove propositions \ref{prop:HJB_sintesis_1} and \ref{prop:HJB_sintesis_2}, we need the following lemmas:
\medskip

\begin{lemma}\label{lemma:ders}
\begin{enumerate}
\item Let $(x_0,t_0)$ such that $\hat t_S(x_0,t_0)<\infty$. Then,
\begin{equation}\label{eq:dxmu}
\begin{split}
\frac{\partial\hat x_S}{\partial x_0} - f(\hat x_S(x_0,t_0))\frac{\partial\hat t_S}{\partial x_0} = \left.\frac{\partial x^{\tildemuI}}{\partial x_0}(t;x_0,t_0)\right|_{t=\hat t_S(x_0,t_0)} =& \,\frac{f(\hat x_S(x_0,t_0))}{f(x_0)} \\
\frac{\partial\hat x_S}{\partial t_0} - f(\hat x_S(x_0,t_0))\frac{\partial\hat t_S}{\partial t_0} = \left.\frac{\partial x^{\tildemuI}}{\partial t_0}(t;x_0,t_0)\right|_{t=\hat t_S(x_0,t_0)} =&\, -f(\hat x_S(x_0,t_0)).
\end{split}
\end{equation} 

\item Let $(x_0,t_0)$ such that $\hat t_{\sigma}(x_0,t_0)<\infty$. Then,
\begin{equation}\label{eq:hit_state_dxsigma}
\frac{\partial \hat x_{\sigma}}{\partial x_0} - (\mu_I+f(\hat x_{\sigma}))\frac{\partial \hat t_{\sigma}}{\partial x_0} \,=\, \frac{\mu_I+f(\hat x_{\sigma}(x_0,t_0))}{\mu_I+f(x_0)} .
\end{equation}
\end{enumerate}
\end{lemma}

\begin{proof}
See Appendix \ref{app:sensitivity}.
\end{proof}
\medskip

\begin{lemma}\label{lemma:dJdx0_C}
For sets $\Theta^{\star}$ and $S^{\star}$ defined in \eqref{eq:Theta} and \eqref{eq:set-last-switch} respectively, one has
\begin{equation*}
\begin{split}
\Theta^{\star} &= \left\{(x_i,t_i)\,|\, C<\frac{\partial J^0}{\partial x_i}(x_i,t_i,T)\right\},\\
S^{\star} &= \left\{(x_i,t_i)\,|\, C=\frac{\partial J^0}{\partial x_i}(x_i,t_i,T)\right\},\\
(\Theta^{\star})^C\backslash S^{\star} &= \left\{(x_i,t_i)\,|\, C>\frac{\partial J^0}{\partial x_i}(x_i,t_i,T)\right\}.
\end{split}
\end{equation*}
\end{lemma}

\begin{proof}

Replacing $w=0$ and $t_f=T$ in \eqref{eq:partialJ_cuerpo}, we get 
\begin{equation}\label{eq:dJ0_dx0}
\frac{\partial J^0}{\partial x_i}(x_i,t_i,T) = \frac{1}{A}\frac{1-e^{-\sqrt{\Delta}(T-t_i)}}{(x_i-\rzeromenos)-(x_i-\rzeromas)e^{-\sqrt{\Delta}(T-t_i)}}.
\end{equation}

Since the denominator of the right-hand side expression in \eqref{eq:dJ0_dx0} is always nonnegative (as long as $x_i>\rzeromenos$ and $t_i\leq T$), a point $(x_i,t_i)$ satisfies $C>\frac{\partial J^0}{\partial x_i}(x_i,t_i,T)$ (resp. $=,<$) if and only if 
\begin{equation}\label{eq:xCstar}
x_i \, >\, x_C^{\star}(t_i) \,:=\,   \frac{1}{AC} + \rzeromas - \frac{\rzeromas-\rzeromenos}{1-e^{-\sqrt{\Delta}(T-t_i)}}.
\end{equation}
(resp. $=,<$). Now,
\begin{equation*}
x_i>x_C^{\star}(t_i) \,\Leftrightarrow\,\left\{\quad \left[\,
t_i>t_S(x_i),\mbox{ if } x_i<\frac{1}{AC}+\rzeromenos \, \right] \quad\mbox{ or }\quad 
x_i\geq \frac{1}{AC}+\rzeromenos\quad \right\}.
\end{equation*}


Indeed, the first part of the right-hand side of last equivalence can be obtained using the alternative form 
\begin{equation*}\label{eq:ts_alternativo}
t_S(x_s) 
= T+\frac{1}{\sqrt{\Delta}}\log\left(1- \frac{AC(\rzeromas-\rzeromenos)}{1+AC(\rzeromas-x_s)} \right),
\end{equation*}
provided that $t_S(\cdot)$ is defined for $x_i<\frac{1}{AC}+\rzeromenos = 2\xbar-\rzeromas$. The second part can be obtained because the function $x_C^{\star}(\cdot)$ is decreasing, then for any $x_i\geq 2\xbar-\rzeromas$ we have $x_C^{\star}(t_i)<\lim_{t\searrow-\infty}x_C^{\star}(t)=2\xbar-\rzeromas \leq x_i$. This proves the Lemma.
\end{proof}
\medskip

\begin{lemma}\label{lemma:dWdx0_C}
Consider, for initial conditions $(x_0,t_0)$ such that $\hat t_S(x_0,t_0)<\infty$ and $x_0>0$, the cost $W_S(x_0,t_0)$ defined in \eqref{eq:WS}. If $x_0\neq \rmumas$, then
\begin{equation*}\label{eq:LS}
C- \frac{\partial W_S}{\partial x_0}(x_0,t_0) < 0\, \,(\mbox{resp.}=0) \quad\mbox{if and only if}\quad \frac{x_0 + \hat x_S(x_0,t_0)}{2} < \xbar \,(\mbox{resp.}=\xbar).
\end{equation*}
\end{lemma}

\begin{proof}
Since $(\hat x_S,\hat t_S)=(\hat x_S(x_0,t_0),\hat t_S(x_0,t_0))$ corresponds to the intersection point (in the plane $(x,t)$) of the curves $t\mapsto(x^{\tildemuI}(t;x_0,t_0),t)$ with $x_s\mapsto(x_s,t_S(x_s))$ (or, in its alternative form from \eqref{eq:xCstar}, $t_s\mapsto(x_C^{\star}(t_s),t_s)$), $(\hat x_S,\hat t_S)$ satisfies the equation
\begin{equation*}\label{eq:ap_tshat}
\hat x_S = \frac{x_0 B/A}{x_0-(x_0-B/A)e^{-B(\hat t_S-t_0)}} = \frac{1}{AC}+\rzeromas-\frac{\sqrt{\Delta}/A}{1-e^{-\sqrt{\Delta}(T-\hat t_S)}}.
\end{equation*}
If $x_0\neq \rmumas$, (where $\rmumas=B/A$ is a root of $f(\cdot)$), we get 
\begin{equation*}
e^{-B(\hat t_S-t_0)} = x_0 \frac{1-\frac{B}{A\hat x_S}}{x_0-B/A}\quad,\quad e^{-\sqrt{\Delta}(T-\hat t_S)} = \frac{\frac{1}{AC}+\rzeromenos-\hat x_S}{\frac{1}{AC}+\rzeromas-\hat x_S}.
\end{equation*}
In such case,
\begin{equation}\label{eq:aux_cuerpo_01}
\begin{split}
\frac{\partial J^{\tildemuI}}{\partial x_i}(x_0,t_0,\hat t_S) =&\, -\frac{x_0-\hat x_S}{f(x_0)}\\[2mm] 
\frac{\partial J^{\tildemuI}}{\partial t_i}(x_0,t_0,\hat t_S) =&\, -\hat x_S-C\tildemuI,\\[2mm]
\frac{\partial J^{0}}{\partial x_i}(\hat x_S,\hat t_S,T) = &\, \left.\frac{1}{A}\frac{1}{\frac{1}{AC}+x_i-\hat x_S}\right|_{x_i=\hat x_S} = C,\\[2mm]
\frac{\partial J^{0}}{\partial t_i}(\hat x_S,\hat t_S,T) = &\, -x_i-\left.\frac{1}{A}\frac{-Ax_i^2+Bx_i+\tildemuI}{\frac{1}{AC}+x_i-\hat x_S}\right|_{x_i=\hat x_S} = -\hat x_S -C(f(\hat x_S)+\tildemuI),
\end{split}
\end{equation}

Denote $\xi_{0S}:=(x_0,t_0,\hat t_S)$ and $\xi_{ST}=(\hat x_S,\hat t_S,T)$. Using \eqref{eq:dxmu} and \eqref{eq:aux_cuerpo_01}, if $x_0\neq B/A (=\rmumas)$, then  $f(x_0)\neq0$, and we have
\begin{equation*}
\begin{split}
\frac{\partial W_S}{\partial x_0}(x_0,t_0) 
=&\, \frac{\partial J^{\tildemuI}(\xi_{0S})}{\partial x_i} + \frac{\partial J^{0}(\xi_{ST})}{\partial x_i}\left.\frac{\partial x^{\tildemuI}}{\partial x_0}\right|_{t=\hat t_S}  \\
&\, + \frac{\partial \hat t_S}{\partial x_0}\left[ -\frac{\partial J^{\tildemuI}(\xi_{0S})}{\partial t_i}  + f(\hat x_S)\frac{\partial J^{0}(\xi_{ST})}{\partial x_i} + \frac{\partial J^{0}(\xi_{ST})}{\partial t_i} \right],\\
=&\, \frac{1}{f(x_0)} (-x_0+\hat x_S+Cf(\hat x_S)).
\end{split}
\end{equation*}

Then,
\begin{equation*}
C- \frac{\partial W_S}{\partial x_0}(x_0,t_0) = 2AC\frac{x_0-\hat x_S(x_0,t_0)}{f(x_0)}\left(\xbar-\frac{x_0+\hat x_S(x_0,t_0)}{2}\right).
\end{equation*}

If $x_0>B/A$ (resp. $<B/A$), then $f(x_0)<0$ (resp. $>0$). Consequently, the term $\frac{x_0-\hat x_S(x_0,t_0)}{f(x_0)}$ is nonpositive. Thus, the result follows.
\end{proof}
\medskip

We now proceed to prove Propositions \ref{prop:HJB_sintesis_1} and \ref{prop:HJB_sintesis_2}.
\medskip

\begin{proof}[Proof of Proposition \ref{prop:HJB_sintesis_1}]
At any point $(x_0,t_0)$ in the sets $\Theta^C\backslash S$ and $\Theta$, the function $W(\cdot)$ is differentiable, but not necessarily on $S$. Also, Lemma \ref{lemma:dWdx0_C} does not provide information on the set $X_0:= \{(x,t)\,|\,x=\rmumas,\,t<t_S(\rmumas)\}$. Notice that $S$ and $X_0$ are one-dimensional manifolds. Then, we split the proof in two possible situations: $(x_0,t_0)\in\Theta^C\backslash S$, $(x_0,t_0)\in\Theta\backslash X_0$.
\begin{enumerate}
\item For $(x_0,t_0)\in\Theta^C\backslash S$, $W(x_0,t_0)=J^{0}(x_0,t_0,T)$. Since $J^{0}(x_0,t_0,T)$ satisfies \eqref{eq:HJ_0}, \eqref{eq:HJB} takes the form 
\begin{equation}\label{eq:HJB_V}
\inf_{w\in[0,\tildemuI]}\left\{\, w\left[C-\frac{\partial J^0}{\partial x_i}(x_0,t_0,T)\right]\, \right\}=0.
\end{equation} 
Thanks to Lemma \ref{lemma:dJdx0_C}, we know that $C-\frac{\partial J^0}{\partial x_i}(x_0,t_0,T)>0$ for $(x_0,t_0)\in\Theta^C\backslash S$, from where we conclude that \eqref{eq:HJB_V} is true, with $w=0$.

\item For $(x_0,t_0)\in\Theta\backslash X_0$, $W(x_0,t_0) = W_S(x_0,t_0)  = J^{\tildemuI}(x_0,t_0,\hat t_S) + J^{0}(\hat x_S,\hat t_S,T)$ (where, for simplicity, we write $(\hat x_S,\hat t_S)=(\hat x_S(x_0,t_0),\hat t_S(x_0,t_0))$). Then, denoting $\xi_{0S}:=(x_0,t_0,\hat t_S)$ and $\xi_{ST}=(\hat x_S,\hat t_S,T)$, and using \eqref{eq:dxmu} and \eqref{eq:aux_cuerpo_01}, since $f(x_0)\neq0$, we have
\begin{equation}\label{eq:derivadas_V_Theta}
\left\{\quad\begin{split}
\frac{\partial W}{\partial x_0}(x_0,t_0) 
=&\, \frac{\partial J^{\tildemuI}(\xi_{0S})}{\partial x_i} + \left.\frac{\partial J^{0}(\xi_{ST})}{\partial x_i}\frac{\partial x^{\tildemuI}}{\partial x_0}\right|_{t=\hat t_S} \\
&\,+ \frac{\partial \hat t_S}{\partial x_0}\left[ -\frac{\partial J^{\tildemuI}(\xi_{0S})}{\partial t_i}  + f(\hat x_S)\frac{\partial J^{0}(\xi_{ST})}{\partial x_i} + \frac{\partial J^{0}(\xi_{ST})}{\partial t_i} \right],\\[2mm]
=&\, \frac{1}{f(x_0)} (-x_0+\hat x_S+Cf(\hat x_S)),\\[2mm]
\frac{\partial W}{\partial t_0}(x_0,t_0) 
=&\, \frac{\partial J^{\tildemuI}(\xi_{0S})}{\partial t_i} +\left. \frac{\partial J^{0}(\xi_{ST})}{\partial x_i}\frac{\partial x^{\tildemuI}}{\partial t_0}\right|_{t=\hat t_S} \\
& \, + \frac{\partial \hat t_S}{\partial t_0}\left[ -\frac{\partial J^{\tildemuI}(\xi_{0S})}{\partial t_i}  + f(\hat x_S)\frac{\partial J^{0}(\xi_{ST})}{\partial x_i} + \frac{\partial J^{0}(\xi_{ST})}{\partial t_i} \right],\\[2mm]
=&\, -\hat x_S-C\tildemuI-Cf(\hat x_S).
\end{split}\right.
\end{equation}

Then, using \eqref{eq:derivadas_V_Theta}, the HJB equation \eqref{eq:HJB} takes the form
\begin{equation}\label{eq:HJB_procesada_cuerpo}
\begin{split}
0 =&\,-\hat x_S-C(\tildemuI+f(\hat x_S)) + x_0 + \frac{\tildemuI+f(x_0)}{f(x_0)} (-x_0+\hat x_S+Cf(\hat x_S)) \\
&\, + \inf_{w\in[0,\tildemuI]}\left\{ \left[ C- \frac{\partial W}{\partial x_0} \right]w \right\}.
\end{split}
\end{equation}

The term that multiplies $w$ in the infimum in \eqref{eq:HJB_procesada_cuerpo} is negative. Indeed, since we are in the case $\xbar\leq\rzeromas$, the asymptote of $t_S(\cdot)$ that defines $S$ is $\xssup=2\xbar-\rzeromas\leq\xbar$. As we are considering $(x_0,t_0)\in\Theta\backslash X_0$, we have $x_0<\xbar$ and $\hat x_S\leq\max\{x_0,\rmumas\}<\xbar$, which implies that $\xbar-\frac{x_0+\hat x_S}{2}>0$. Lemma \ref{lemma:dWdx0_C} leads to $C-\frac{\partial W}{\partial x_0}(x_0,t_0)<0$.

Then, the infimum in \eqref{eq:HJB_procesada_cuerpo} is attained with $w=\tildemuI$. Thus, \eqref{eq:HJB_procesada_cuerpo} becomes
\begin{equation*}\label{eq:HJB_procesada_cuerpo2}
\begin{split}
-\hat x_S-C\tildemuI-Cf(\hat x_S) + x_0 + & \, \frac{\tildemuI+f(x_0)}{f(x_0)} (-x_0+\hat x_S+Cf(\hat x_S))\\[2mm]
&\quad + \left[ C-\frac{1}{f(x_0)} (-x_0+\hat x_S+Cf(\hat x_S)) \right]\tildemuI =0.
\end{split}
\end{equation*}
which shows that \eqref{eq:HJB_procesada_cuerpo} or, equivalently, \eqref{eq:HJB}, is satisfied, with $w=\tildemuI$.

\end{enumerate}

We conclude the result by Theorem \ref{teo:bresanpiccoli}.
\end{proof}

\begin{proof}[Proof of Proposition \ref{prop:HJB_sintesis_2}]
At any point $(x_0,t_0)$ in the interior of the sets $\mathcal V$, $\mathcal S$, $\mathcal T$, the function $W(\cdot)$ is differentiable, but not necessarily on the boundary of $\mathcal S$. Also, as before, Lemma \ref{lemma:dWdx0_C} does not provide information on the set $X_0:= \{(x,t)\,|\,x=\rmumas,\,t<t_S(\rmumas)\}$. We split the proof in three possible situations: $(x_0,t_0)\in\mathcal V\backslash (S\cup\Gamma)$, $(x_0,t_0)\in\mathcal S\backslash X_0$, $(x_0,t_0)\in\mathcal T\backslash\sigma$.
		
\begin{enumerate}
\item Suppose $(x_0,t_0)\in\mathcal V\backslash (S\cup\Gamma)$. Then, $W(x_0,t_0)=J^0(x_0,t_0,T)$. Notice that $\mathcal V\backslash (S\cup\Gamma)\subseteq (\Theta^{\star})^C\backslash S^{\star}$. Thus, from Lemma \ref{lemma:dJdx0_C}, we have that $C>\frac{\partial W}{\partial x_0}(x_0,t_0)$. Thus, $w=0$ is the minimizer in \eqref{eq:HJB}. This, with \eqref{eq:HJ_0}, implies that \eqref{eq:HJB} is satisfied (similar to the arguments used in point 1. of Proposition \ref{prop:HJB_sintesis_1}).

\item If $(x_0,t_0)\in\mathcal S\backslash X_0$, then $W(x_0,t_0)=W_S(x_0,t_0)$.   

Suppose $x_0<\rmumas$. As in the region $\{(x,t)\,|\,x<\rmumas\}$ we have $f(x)>0$, with $\rmumas$ equilibrium of the dynamics, then $x_0<x_{S}(x_0,t_0)<\rmumas<\xbar$. This implies that $\frac{x_0+\hat x_S(x_0,t_0)}{2}<\xbar$. 

Now, suppose $x_0>\rmumas$. Notice that in the region $\{(x,t)\,|\,x>\rmumas\}$, we have $f(x)<0$, with $\rmumas$ equilibrium of the dynamics. Then, there exists $(x_{\sigma},t_{\sigma})\in\sigma$ that belongs to the same trajectory than $(x_0,t_0)$ (controlled by $w=\mu_I$), with $t_0>t_{\sigma}$ and $x_0<x_{\sigma}$. Thus, $\hat x_{S}(x_0,t_0)=\hat x_S(x_{\sigma},t_{\sigma})<x_0<x_{\sigma}$. Moreover, by construction of the curve $\sigma$, $x_{\sigma}$ and $\hat x_S(x_{\sigma},t_{\sigma})$ are linked by the relation $\frac{x_{\sigma} + \hat x_S(x_{\sigma},t_{\sigma})}{2}=\xbar$. In summary, 
\begin{equation*}
\frac{x_0 + \hat x_S(x_0,t_0)}{2}< \frac{x_{\sigma} + \hat x_S(x_0,t_0)}{2}\leq \frac{x_{\sigma} + \hat x_S(x_{\sigma},t_{\sigma})}{2}=\xbar.
\end{equation*}

In both cases, by Lemma \ref{lemma:dWdx0_C} we have $C-\frac{\partial W}{\partial x_0}(x_0,t_0)<0$. We conclude by the same arguments than in point 2. of Proposition \ref{prop:HJB_sintesis_1}. 

\item Supose $(x_0,t_0)\in\mathcal T\backslash\sigma$. For simplicity, write $(\hat x_{\sigma},\hat t_{\sigma})=(\hat x_{\sigma}(x_0,t_0),\hat t_{\sigma}(x_0,t_0))$, and $(\hat x_{S},\hat t_{S})=(\hat x_{S}(\hat x_{\sigma},\hat t_{\sigma}),\hat t_{S}(\hat x_{\sigma},\hat t_{\sigma}))$. Denoting $\xi_{0\sigma}:=(x_0,t_0,\hat t_{\sigma})$, $\xi_{\sigma S}=(\hat t_{\sigma},\hat x_{\sigma},\hat t_S)$  $\xi_{ST}=(\hat x_S,\hat t_S,T)$, we have for any $(x_0,t_0)\in\mathcal T\backslash\sigma$
\begin{equation}\label{eq:derivadas_V_T}
\left\{\quad\begin{split}
\frac{\partial W}{\partial x_0}(x_0,t_0) \,=&\, \frac{\partial J^0(\xi_{0\sigma})}{\partial x_i} + U_1\frac{\partial \hat t_{\sigma}}{\partial x_0} + U_2\frac{\partial \hat x_{\sigma}}{\partial x_0},\\[2mm]
\frac{\partial W}{\partial t_0}(x_0,t_0) \,=&\, \frac{\partial J^0(\xi_{0\sigma})}{\partial t_i} + U_1\frac{\partial \hat t_{\sigma}}{\partial t_0} + U_2\frac{\partial \hat x_{\sigma}}{\partial t_0}.
\end{split}\right.
\end{equation}
with 
\begin{equation}\label{eq:W1W2}
\left\{\quad\begin{split}
U_1 :=&\, -\frac{\partial J^{0}(\xi_{0\sigma})}{\partial t_i} + \frac{\partial J^{\tildemuI}(\xi_{\sigma S})}{\partial t_i} - \frac{\partial J^{\tildemuI}(\xi_{\sigma S})}{\partial t_i}\frac{\partial \hat t_S}{\partial t}(\hat x_{\sigma},\hat t_{\sigma}) \\[2mm]
&\quad + \frac{\partial J^{0}(\xi_{ST})}{\partial x_i}\frac{\partial \hat x_S}{\partial t}(\hat x_{\sigma},\hat t_{\sigma}) + \frac{\partial J^{0}(\xi_{ST})}{\partial t_i}\frac{\partial \hat t_S}{\partial t}(\hat x_{\sigma},\hat t_{\sigma}),\\[2mm]
U_2 :=&\, \frac{\partial J^{\tildemuI}(\xi_{\sigma S})}{\partial x_i} - \frac{\partial J^{\tildemuI}(\xi_{\sigma S})}{\partial t_i}\frac{\partial \hat t_S}{\partial x}(\hat x_{\sigma},\hat t_{\sigma}) + \frac{\partial J^{0}(\xi_{ST})}{\partial x_i}\frac{\partial \hat x_S}{\partial x}(\hat x_{\sigma},\hat t_{\sigma}) \\[2mm]
&\quad + \frac{\partial J^{0}(\xi_{ST})}{\partial t_i}\frac{\partial \hat t_S}{\partial x}(\hat x_{\sigma},\hat t_{\sigma}).
\end{split}\right.
\end{equation}

From \eqref{eq:aux_cuerpo_01}, \eqref{eq:W1W2} becomes
\begin{equation}\label{eq:W1W2_new}
\left\{\,\begin{split}
U_1 =&\, -\frac{\partial J^{0}(\xi_{0\sigma})}{\partial t_i} -(\hat x_S+C\mu_I) + C\left(\frac{\partial \hat x_S}{\partial t}(\hat x_{\sigma},\hat t_{\sigma}) - f(\hat x_S)\frac{\partial \hat t_S}{\partial t}(\hat x_{\sigma},\hat t_{\sigma})\right),\\[2mm]
U_2 =&\, -\frac{\hat x_{\sigma}-\hat x_S}{f(\hat x_{\sigma})} +C\left(\frac{\partial \hat x_S}{\partial x}(\hat x_{\sigma},\hat t_{\sigma}) -f(\hat x_S)\frac{\partial \hat t_S}{\partial x}(\hat x_{\sigma},\hat t_{\sigma})\right).
\end{split}\right.
\end{equation}

From \eqref{eq:dxmu}, evaluating in $(\hat x_{\sigma},\hat t_{\sigma})$,
\begin{equation*}\label{eq:dxt_dx0}
\begin{split}
\frac{\partial \hat x_S}{\partial x}(\hat x_{\sigma},\hat t_{\sigma}) -f(\hat x_S)\frac{\partial \hat t_S}{\partial x}(\hat x_{\sigma},\hat t_{\sigma}) =&\, \frac{f(\hat x_{S})}{f(\hat x_{\sigma})},\\
\frac{\partial \hat x_S}{\partial t}(\hat x_{\sigma},\hat t_{\sigma}) -f(\hat x_S)\frac{\partial \hat t_S}{\partial t}(\hat x_{\sigma},\hat t_{\sigma}) =&\, -f(\hat x_{S}),
\end{split}
\end{equation*}
which simplifies \eqref{eq:W1W2_new} to
\begin{equation}\label{eq:W1W2_nnew}
U_1 \,=\, -\frac{\partial J^{0}(\xi_{0\sigma})}{\partial t_i} -(\hat x_S + Cf(\hat x_S) + C\mu_I),\quad
U_2 \,=\, \frac{\hat x_S - \hat x_{\sigma} + Cf(\hat x_{S})}{f(\hat x_{\sigma})}.
\end{equation}

Since $\hat x_{\sigma}$ is constructed from the expression $\lambda_{\sup}(\hat x_{\sigma},\hat x_S)=C$ (see Appendix \ref{app:solutionLogistic}), then
\begin{equation*}\label{eq:fxs_fxsigma}
\hat x_S + Cf(\hat x_S) = \hat x_{\sigma} + Cf(\hat x_{\sigma}),
\end{equation*}
which transforms \eqref{eq:W1W2_nnew} to
\begin{equation}\label{eq:W1W2_nnnew}
U_1 \,=\, -\frac{\partial J^{0}(\xi_{0\sigma})}{\partial t_i} -(\hat x_{\sigma} + Cf(\hat x_{\sigma}) + C\mu_I ),\quad
U_2 \,=\, C.
\end{equation}

Replacing \eqref{eq:W1W2_nnnew} in \eqref{eq:derivadas_V_T},
\begin{equation}\label{eq:d_V_x0_2}
\begin{split}
\frac{\partial W}{\partial x_0}(x_0,t_0) 
=&\, \frac{\partial J^0(\xi_{0\sigma})}{\partial x_i} + C\left[ \frac{\partial \hat x_{\sigma}}{\partial x_0} - (\mu_I+ f(\hat x_{\sigma}))\frac{\partial \hat t_{\sigma}}{\partial x_0}  \right] \\
&\,- \frac{\partial \hat t_{\sigma}}{\partial x_0}\left[ \frac{\partial J^{0}(\xi_{0\sigma})}{\partial t_i} + \hat x_{\sigma}   \right] .
\end{split}
\end{equation}

Replacing \eqref{eq:hit_state_dxsigma} (from Lemma \ref{lemma:ders}) in \eqref{eq:d_V_x0_2}, and using \eqref{eq:HJ_0}, we obtain 
\begin{equation}\label{eq:d_V_x0_3}
\begin{split}
\frac{\partial W}{\partial x_0}(x_0,t_0) 
=&\, -\frac{x_0-\hat x_{\sigma}}{\mu_I+f(x_0)}  + C\frac{\mu_I+f(\hat x_{\sigma})}{\mu_I+f(x_0)} \\
&\,- \left[ \frac{1}{\mu_I + f(x_0)} + \frac{\partial \hat t_{\sigma}}{\partial x_0}\right]\left[ \frac{\partial J^{0}(\xi_{0\sigma})}{\partial t_i} + \hat x_{\sigma}  \right].
\end{split}
\end{equation}

Notice that, from \eqref{eq:partialJ_cuerpo} and \eqref{eq:solucion_logistica} (with $w=0$), we have
\begin{equation*}
\frac{\partial J^{0}(\xi_{0\sigma})}{\partial t_i} + \hat x_{\sigma} \,=\, 0,
\end{equation*}
which transforms \eqref{eq:d_V_x0_3} into 
\begin{equation*}\label{eq:d_V_x0_4}
\frac{\partial W}{\partial x_0}(x_0,t_0) 
\,=\, -\frac{x_0-\hat x_{\sigma}-C(\mu_I+f(\hat x_{\sigma}))}{\mu_I+f(x_0)}.
\end{equation*}

We prove that, in $\mathcal T\backslash\sigma$, $C>\frac{\partial W}{\partial x_0}$. Indeed, for $(x_0,t_0)\in\mathcal T\backslash\sigma$ we have $x_0>\rzeromas$, which implies $\mu_I+f(x_0)<0$. Then, $C>\frac{\partial W}{\partial x_0}$ if and only if 
\begin{equation*}
x_0-\hat x_{\sigma} - C(f(\hat x_{\sigma})-f(x_0)) \,=\, 2AC(x_0-\hat x_{\sigma})\left( \xbar - \frac{\hat x_{\sigma}+x_0}{2} \right) < 0,
\end{equation*}
which is true, since for $(x_0,t_0)\in\mathcal T\backslash\sigma$ we have $\xbar<\hat x_{\sigma}<x_0$. Thus, the Hamiltonian is minimized by $w=0$. From \eqref{eq:edp_tS}, \eqref{eq:edp_xS} (in Appendix \ref{app:sensitivity}), and \eqref{eq:HJ_0},
\begin{equation*}
\begin{split}
\frac{\partial W}{\partial t_0} + x_0 + (\tildemuI+f(x_0))\frac{\partial W}{\partial x_0} \,=\,&\, \frac{\partial J^{0}(\xi_{0\sigma})}{\partial t_i} + x_0 + (\tildemuI+f(x_0))\frac{\partial J^{0}(\xi_{0\sigma})}{\partial x_i}\\
&\, + U_1\left[ \frac{\partial \hat t_{\sigma}}{\partial t_0} + (\tildemuI+f(x_0))\frac{\partial \hat t_{\sigma}}{\partial x_0} \right]\\
&\, + U_2\left[ \frac{\partial \hat x_{\sigma}}{\partial t_0} + (\tildemuI+f(x_0))\frac{\partial \hat x_{\sigma}}{\partial x_0} \right],\\
\,=\,&\, 0,
\end{split}
\end{equation*}
which proves the result.

\end{enumerate}

We conclude the result by Theorem \ref{teo:bresanpiccoli}.
\end{proof}
\medskip

\begin{remark}\label{remark:xssup}
Depending on the parameters of the problem, it is possible that the supremum of the states of last switch $\xssup$ does not belong to the interval $[0,1]$. This impacts directly in the qualitative behavior of the optimal controls. Indeed
\begin{enumerate}
\item If $x_{S}^{\sup}\leq0$ (possible if $\xbar\leq \rzeromas/2$; see point 1. of Corollary \ref{lemma:xs_xhash} and point 3. of Lemma \ref{lemma:xs_xt}), the curve $S$ is contained in the zone $\{(x,t)\,|\,x<0\}$. In such case, there are no switches along any optimal trajectories, and the optimal control is always $w=0$ (see Proposition \ref{prop:HJB_sintesis_1}).
\item  If $x_S^{\sup}>1$ (possible if $\xbar\geq1$; see point 3. of Corollary \ref{lemma:xs_xhash} and point 3. of Lemma \ref{lemma:xs_xt}), the curve $\sigma$ exists, and is contained in the zone $\{(x,t)\,|\,x>1\}$. Then, there is at most one switch for every optimal trajectory, from $w=\mu_I$ to $w=0$ (see Proposition \ref{prop:HJB_sintesis_2}).
\end{enumerate}
\end{remark}

Propositions \ref{prop:HJB_sintesis_1} and \ref{prop:HJB_sintesis_2}, provide the necessary information to drive the process optimally. Suppose $(x_0,t_0)\in[0,1]\times[0,T]$. Then, we need to:

\begin{enumerate}
\item Compute the steady states (globally asymptotic stable) $\rmumas$ (see \eqref{eq:rmui}) and $\rzeromas$ (see \eqref{eq:rzero}) associated to constant controls $w(\cdot)\equiv\mu_I$ and $w(\cdot)\equiv 0$ respectively. Compute the point $\xbar$ given by \eqref{eq:def_xbar}. Compute $x_S^{\sup}$, as in Corollary \ref{lemma:xs_xhash}.

\item If $x_{S}^{\sup}\leq0$, there are no switches. Thus, use $w=0$ until time $T$.

\item If $\xbar\leq \rzeromas$, compute the curve of last switch $S$ defined in \eqref{eq:set-last-switch}  and consider the set $\Theta$, given by \eqref{eq:Theta}, that lies below $S$:
\begin{enumerate}
\item If $(x_0,t_0)\in\Theta$, use $w=\mu_I$ until the trajectory intersects $S$;
\item If $(x_0,t_0)\in\Theta^{C}$, use $w=0$ until time $T$.
\end{enumerate}

\item If $\rzeromas<\xbar$, compute  the curve of last switch $S$ defined in \eqref{eq:set-last-switch} and the curve of the first switch $\sigma$ defined in \eqref{eq:Sigma}. Consider the sets  $\mathcal V$, $\mathcal T$, and $ \mathcal S$  introduced in Proposition \ref{prop:HJB_sintesis_2} (see also Figure \ref{fig:systhesis}):
\begin{enumerate}
\item If $(x_0,t_0)\in\mathcal T$, use $w=0$ until the trajectory intersects $\sigma$;
\item If $(x_0,t_0)\in \mathcal S$, use $w=\mu_I$ until the trajectory intersects $S$;
\item If $(x_0,t_0)\in\mathcal V\cup S\cup\Gamma$, use $w=0$ until time $T$. 
\end{enumerate}
\end{enumerate}


\section{Conclusions}\label{sec:conclusions}

In this paper, we have proposed a simple epidemiological model for representing the spread of a disease in a prison, and its control through a screening and treatment policy at the entry point of new inmates. We obtain the complete synthesis of solutions that minimize the total cost (screening/treatment at the entry point plus the cost associated to have infected individuals inside the prison). Our analysis allows to conclude that optimal solutions are of bang-bang type, that is, it will be optimal to apply a full screening strategy (i.e., to all new inmates) or to suppress this strategy in some periods of time. Also, we determine that there are at most two switching instants, finishing always with a null-screening policy. Therefore, the optimal policies can be only: (i) Null-screening strategy; (ii) Full-screening strategy and then the null-screening strategy; (iii) Null-screening strategy, full-screening strategy and then the null-screening strategy. The different cases are determined in terms of the problem data: (a) Disease: contagious and recovery rates; (b) Prison: inmate entry and exit rates, considered equals; (c) Costs: unitary costs of screening/treatment at the entry point and the unitary cost associated to have infected individuals inside the prison. The switching curves in the state-time space are computed explicitly, allowing to obtain feedback rules for all cases. In addition, explicit expressions of the value function are obtained, making possible to estimate optimal costs. 


\section{Acknowledgments}
This work was funded by FONDECYT grants N 1200355 (first author) and N 3180367 (second author), and PFCHA/Doctorado Becas Chile grant N 2017-21171813 (third author), all programs from ANID-Chile. The authors are very grateful to Professors Térence Bayen (Université d'Avignon, France) and Cristopher Hermosilla (Universidad Técnica Federico Santa María, Chile)   for fruitful discussions and insightful ideas.


\bibliographystyle{siamplain}
\bibliography{preprint_oc_sis_prison_2020-10-31}

\begin{thebibliography}{10}

\bibitem{BarCapBook}
{\sc M.~Bardi and I.~Capuzzo-Dolcetta}, {\em Optimal control and viscosity
  solutions of {H}amilton-{J}acobi-{B}ellman equations}, Systems \& Control:
  Foundations \& Applications, Birkh{\"a}user Boston Inc., Boston, MA, 1997,
  \url{http://dx.doi.org/10.1007/978-0-8176-4755-1}.

\bibitem{BayenetAl2015}
{\sc T.~Bayen, M.~Mazade, and F.~Mairet}, {\em Analysis of an optimal control
  problem connected to bioprocesses involving a saturated singular arc},
  Discrete Contin. Dyn. Syst. Ser. B, 20 (2015), pp.~39--58,
  \url{https://doi.org/10.3934/dcdsb.2015.20.39}.

\bibitem{belenko2008}
{\sc S.~Belenko, R.~Dembo, D.~Weiland, M.~Rollie, C.~Salvatore, A.~Hanlon, and
  K.~Childs}, {\em Recently arrested adolescents are at high risk for sexually
  transmitted diseases}, Sexually transmitted diseases, 35 (2008),
  pp.~758--763.

\bibitem{blandford2007}
{\sc J.~M. Blandford, T.~L. Gift, S.~Vasaikar, D.~Mwesigwa-Kayongo, P.~Dlali,
  and R.~N. Bronzan}, {\em Cost-effectiveness of on-site antenatal screening to
  prevent congenital syphilis in rural eastern cape province, republic of south
  africa}, Sexually transmitted diseases, 34 (2007), pp.~S61--S66.

\bibitem{boelaert2007}
{\sc M.~Boelaert, S.~Bhattacharya, F.~Chappuis, S.~H. El~Safi, A.~Hailu,
  D.~Mondal, S.~Rijal, S.~Sundar, M.~Wasunna, and R.~W. Peeling}, {\em
  Evaluation of rapid diagnostic tests: visceral leishmaniasis}, Nature Reviews
  Microbiology, 5 (2007), pp.~S31--S39.

\bibitem{BoscainPiccoli}
{\sc U.~Boscain and B.~Piccoli}, {\em Optimal syntheses for control systems on
  2-{D} manifolds}, vol.~43 of Math\'{e}matiques \& Applications (Berlin)
  [Mathematics \& Applications], Springer-Verlag, Berlin, 2004.

\bibitem{BrauerCCC2013}
{\sc F.~Brauer and C.~Castillo-Chavez}, {\em Mathematical Models for
  Communicable Diseases}, vol.~84 of CBMS-NSF Regional Conference Series in
  Applied Mathematics, SIAM, Philadelphia, 2013.

\bibitem{BressanPiccoli2007}
{\sc A.~Bressan and B.~Piccoli}, {\em Introduction to the mathematical theory
  of control}, vol.~2 of AIMS Series on Applied Mathematics, American Institute
  of Mathematical Sciences (AIMS), Springfield, MO, 2007.

\bibitem{browder1998ordinary}
{\sc A.~Browder, W.~Walter, R.~Thompson, W.~Wolfgang, W.~Walter, S.~Axler,
  F.~Gehring, and P.~Halmos}, {\em Ordinary Differential Equations}, Graduate
  Texts in Mathematics, Springer New York, 1998.

\bibitem{campbell1976}
{\sc S.~L. Campbell}, {\em Optimal control of autonomous linear processes with
  singular matrices in the quadratic cost functional}, SIAM J. Control Optim.,
  14 (1976), pp.~1092--1106, \url{https://doi.org/10.1137/0314068}.

\bibitem{castillo2020}
{\sc C.~Castillo-Laborde, P.~Gajardo, M.~N{\'a}jera-De~Ferrari, I.~Matute,
  M.~Hirmas-Adauy, P.~Aguirre, H.~Ram{\'\i}rez, D.~Ram{\'i}rez, and
  X.~Aguilera}, {\em Modelling cost-effectiveness of syphilis detection
  strategies in prisoners: Exploratory exercise in a chilean male prison},
  Preprint Research Square,  (2020),
  \url{https://www.researchsquare.com/article/rs-54180/v1}.

\bibitem{Clark1979}
{\sc C.~W. Clark and J.~D. De~Pree}, {\em A simple linear model for the optimal
  exploitation of renewable resources}, Appl. Math. Optim., 5 (1979),
  pp.~181--196, \url{https://doi.org/10.1007/BF01442553}.

\bibitem{clarke}
{\sc F.~Clarke}, {\em Functional Analysis, Calculus of Variations and Optimal
  Control}, Springer, 2013.

\bibitem{coddington1955theory}
{\sc A.~Coddington and N.~Levinson}, {\em Theory of ordinary differential
  equations}, International series in pure and applied mathematics,
  McGraw-Hill, 1955.

\bibitem{dmitruk2020}
{\sc A.~V. Dmitruk and A.~K. Vdovina}, {\em Study of a one-dimensional optimal
  control problem with a purely state-dependent cost}, Differ. Equ. Dyn. Syst.,
  28 (2020), pp.~133--151, \url{https://doi.org/10.1007/s12591-016-0306-x}.

\bibitem{who2007}
{\sc A.~Gatherer, R.~J{\"u}rgens, and H.~St{\"o}ver}, {\em Health in prisons: a
  WHO guide to the essentials in prison health}, WHO Regional Office Europe,
  2007.

\bibitem{gianino2007}
{\sc M.~Gianino, I.~Dal~Conte, K.~Sciol\'e, M.~Galzerano, L.~Castelli,
  R.~Zerbi, I.~Arnaudo, G.~Di~Perri, and G.~Renga}, {\em Performance and costs
  of a rapid syphilis test in an urban population at high risk for sexually
  transmitted infections}, Journal of preventive medicine and hygiene, 48
  (2007), pp.~118--122, \url{http://europepmc.org/abstract/MED/18557305}.

\bibitem{Goh}
{\sc B.~Goh}, {\em Necessary conditions for singular extremals involving
  multiple controls}, SIAM J. on Control, 4 (1966), pp.~717--731.

\bibitem{kelley}
{\sc H.~Kelley, R.~Kopp, and H.~Moyer}, {\em Singular extremals}, in Topics in
  Optimization, L.~G., ed., Academic Press, 1967, pp.~63--101.

\bibitem{khan2009incarceration}
{\sc M.~R. Khan, I.~A. Doherty, V.~J. Schoenbach, E.~M. Taylor, M.~W. Epperson,
  and A.~A. Adimora}, {\em Incarceration and high-risk sex partnerships among
  men in the united states}, Journal of Urban Health, 86 (2009), pp.~584--601.

\bibitem{MariaetAl2011}
{\sc M.~R. Khan, M.~W. Epperson, P.~Mateu-Gelabert, M.~Bolyard, M.~Sandoval,
  and S.~R. Friedman}, {\em Incarceration, sex with an sti- or hiv-infected
  partner, and infection with an sti or hiv in bushwick, brooklyn, ny: A social
  network perspective}, American Journal of Public Health, 101 (2011),
  pp.~1110--1117, \url{https://doi.org/10.2105/AJPH.2009.184721}.
\newblock PMID: 21233443.

\bibitem{kouyoumdjian2012}
{\sc F.~Kouyoumdjian, D.~Leto, S.~John, H.~Henein, and S.~Bondy}, {\em A
  systematic review and meta-analysis of the prevalence of chlamydia,
  gonorrhoea and syphilis in incarcerated persons}, International journal of
  STD \& AIDS, 23 (2012), pp.~248--254.

\bibitem{lee2010}
{\sc S.~Lee, G.~Yearwood, G.~Guillon, L.~Kurtz, M.~Fischl, T.~Friel, C.~Berne,
  and K.~Kardos}, {\em Evaluation of a rapid, point-of-care test device for the
  diagnosis of hepatitis c infection}, Journal of Clinical Virology, 48 (2010),
  pp.~15--17.

\bibitem{lee2011}
{\sc S.~R. Lee, K.~W. Kardos, E.~Schiff, C.~A. Berne, K.~Mounzer, A.~T. Banks,
  H.~A. Tatum, T.~J. Friel, M.~P. DeMicco, W.~M. Lee, et~al.}, {\em Evaluation
  of a new, rapid test for detecting hcv infection, suitable for use with blood
  or oral fluid}, Journal of virological methods, 172 (2011), pp.~27--31.

\bibitem{ejoce2}
{\sc L.~Li, C.~Sun, and J.~Jia}, {\em Optimal control of a delayed {SIRC}
  epidemic model with saturated incidence rate}, Optimal Control Appl. Methods,
  40 (2019), pp.~367--374, \url{https://doi.org/10.1002/oca.2482}.

\bibitem{world2016}
{\sc W.~H. Organization et~al.}, {\em Global health sector strategy on sexually
  transmitted infections 2016-2021: toward ending stis}, tech. report, World
  Health Organization, 2016,
  \url{https://apps.who.int/iris/bitstream/handle/10665/246296/WHO-RHR-16.09-eng.pdf}.

\bibitem{pascoe2009}
{\sc S.~J. Pascoe, L.~F. Langhaug, J.~Mudzori, E.~Burke, R.~Hayes, and F.~M.
  Cowan}, {\em Field evaluation of diagnostic accuracy of an oral fluid rapid
  test for hiv, tested at point-of-service sites in rural zimbabwe}, AIDS
  patient care and STDs, 23 (2009), pp.~571--576.

\bibitem{peeling2009}
{\sc R.~W. Peeling}, {\em Utilisation of rapid tests for sexually transmitted
  infections: promises and challenges}, The Open Infectious Diseases Journal, 3
  (2009), pp.~156--163.

\bibitem{peeling2006}
{\sc R.~W. Peeling, K.~K. Holmes, D.~Mabey, and A.~Ronald}, {\em Rapid tests
  for sexually transmitted infections (stis): the way forward}, Sexually
  Transmitted Infections, 82 (2006), pp.~v1--v6,
  \url{https://doi.org/10.1136/sti.2006.024265}.

\bibitem{peypouquet2015convex}
{\sc J.~Peypouquet}, {\em Convex Optimization in Normed Spaces: Theory, Methods
  and Examples}, SpringerBriefs in Optimization, Springer International
  Publishing, 2015, \url{https://books.google.cl/books?id=nLeTBwAAQBAJ}.

\bibitem{pontryagin}
{\sc L.~S. Pontryagin, V.~G. Boltyanskii, R.~V. Gamkrelidze, and E.~F.
  Mishchenko}, {\em The mathematical theory of optimal processes}, Translated
  by D. E. Brown, A Pergamon Press Book. The Macmillan Co., New York, 1964.

\bibitem{RapaportCartigny2004}
{\sc A.~Rapaport and P.~Cartigny}, {\em Turnpike theorems by a value function
  approach}, ESAIM Control Optim. Calc. Var., 10 (2004), pp.~123--141,
  \url{https://doi.org/10.1051/cocv:2003039}.

\bibitem{RapaportCartigny2005}
{\sc A.~Rapaport and P.~Cartigny}, {\em Competition between most rapid approach
  paths: necessary and sufficient conditions}, J. Optim. Theory Appl., 124
  (2005), pp.~1--27, \url{https://doi.org/10.1007/s10957-004-6463-z}.

\bibitem{victor:ode}
{\sc V.~Riquelme}, {\em {A Hamilton-Jacobi approach of sensitivity of ODE flows
  and switching points in optimal control problems}}, 2020,
  \url{https://arxiv.org/abs/2009.05667}.

\bibitem{Rockafellar1987}
{\sc R.~T. Rockafellar}, {\em Linear-quadratic programming and optimal
  control}, SIAM J. Control Optim., 25 (1987), pp.~781--814,
  \url{https://doi.org/10.1137/0325045}.

\bibitem{sontag1998}
{\sc E.~D. Sontag}, {\em Mathematical control theory}, vol.~6 of Texts in
  Applied Mathematics, Springer-Verlag, New York, second~ed., 1998,
  \url{https://doi.org/10.1007/978-1-4612-0577-7}.
\newblock Deterministic finite-dimensional systems.

\bibitem{ejoce3}
{\sc B.~Stephenson, C.~Lanzas, S.~Lenhart, and J.~Day}, {\em Optimal control of
  vaccination rate in an epidemiological model of {\it {c}lostridium difficile}
  transmission}, J. Math. Biol., 75 (2017), pp.~1693--1713,
  \url{https://doi.org/10.1007/s00285-017-1133-6}.

\bibitem{ejoce1}
{\sc H.~Zhao, P.~Wu, and S.~Ruan}, {\em Dynamic analysis and optimal control of
  a three-age-class {HIV}/{AIDS} epidemic model in china}, Discrete Contin.
  Dyn. Syst. Ser. B, 25 (2020), pp.~3491--3521,
  \url{https://doi.org/10.3934/dcdsb.2020070}.

\end{thebibliography}


\appendix

\section{Appendix}

\if{
\subsection{Existence of optimal solutions}\label{ap:existenceProof}

First let us write the result as exposed in \cite{clarke}. We have the following optimal control problem: Minimize the following functional:

\begin{equation}\label{eq:funcionalGeneral}
J(x,u) = l(x(a),x(b)) + \int_{a}^{b} \Lambda(t,x(t),u(t)) dt,
\end{equation}
subject to:
\begin{equation}\label{eq:dinamicaGeneral}
\left\{\quad\begin{split}
&\dot{x}(t) = g_0(t,x(t)) + \sum_{j=1}^m g_j(t,x(t)) u^j(t) \quad a.e.\\
&u(t) \in U(t) \quad a.e.\\
&(t,x(t)) \in Q \quad \forall t\in[a,b]\\
&(x(a),x(b)) \in E.
\end{split}\right.
\end{equation}
The existence theorem states:

\begin{theorem}\label{ap:TeoremaExistencia}
Let the data of \eqref{eq:funcionalGeneral},\eqref{eq:dinamicaGeneral} satisfy the following hypotheses:

\begin{enumerate}\itemsep0em 
\item Each $g_j$, $(j=0,1,\cdots,m)$ is measurable in $t$, continuous in $x$ and has linear growth: there exists a constant $M$ such that $(t,x)\in Q \Rightarrow |g_j(t,x)|\leq M(1+|x|)$;
\item For almost every $t$, the set $U(t)$ is closed and convex;
\item The sets $E$ and $Q$ are closed and $l:\mathbb{R}^n\times \mathbb{R}^n \rightarrow \mathbb{R}$ is lower semicontinuous. 
\item The running cost $\Lambda(t,x,u)$ is LB-measurable in $t$ and $(x,u)$, and lower semicontinuous in $(x,u):\Lambda(t,x,\cdot)$ is convex for each $(t,x) \in Q$; there is a constant $\lambda_0$ such that $(t,x) \in Q, u\in U(t) \Rightarrow \Lambda(t,x,u)\geq \lambda_0$.
\item The projection $\{ \alpha \in \mathbb{R}^n:(\alpha,\beta)\in E$, for some $ \beta \in \mathbb{R}^n \}$ of $E$ is bounded.
\item One of the following holds for some $r>1$:
\begin{enumerate}
\item There exists $k \in L^r(a,b)$ such that, for almost every $t$, $ u\in U(t) \Rightarrow |u|\leq k(t)$,
\item or there exists $\alpha >0 $ and $\beta$ such that $(t,x)\in Q, u\in U(t) \Rightarrow \Lambda(t,x,u) \geq \alpha|u|^r + \beta$.
\end{enumerate}
\end{enumerate}
Then, if there is at least one admissible process $(x,u)$ for which $J(x,u)$ is finite, the problem admits a solution.
\end{theorem}

Given the previous theorem, we now prove the result for Problem \ref{eq:problem}. Let us recall that given a function $x:\mathbb{R}\rightarrow \mathbb{R}$, and a measurable control function $w:[0,+\infty)\rightarrow [0,\mu_I]$ and $c>0$, the problem is to minimize the cost functional
\begin{equation}\label{eq:funcionalParticular}
J(x,w) = \int_{0}^{T} \left(\, C w(t) + x(t) \,\right) dt,
\end{equation}
subject to
\begin{equation}\label{eq:dinamicaParticular}
\begin{split}
&\dot{x}(t) = \mu_I - w(t) + f(x)
\end{split}
\end{equation}
where $f:\mathbb{R}\rightarrow \mathbb{R}$ is a continuously differentiable function, strictly concave. Also, $f(0) = 0$ and $f(1)<u(t)-\mu_I$, then, if $x(0) \in [0,1]$, then $x(T) \in [0,1]$. Using the same notation as indicated in \cite{clarke}, we have $l(x(a),x(b))=0$, $\Lambda(t,x(t),w(t)) = C w(t) + x(t)$, $g_0(t,x(t)) = \mu_I+f(x(t))$, $g_1(t,x(t)) = -1$, $U(t) = [0,\tildemuI]$, $Q = [0,T]\times[0,1]$, $E = [0,1]\times[0,1]$.
%

Now, let us check each of the conditions in Theorem \ref{ap:TeoremaExistencia}.

\begin{enumerate}\itemsep0em
\item $g_0$ and $g_1$ are continuous in $t$ and $x$. Given that $f$ is continuously differentiable and well defined in all its domain, then in $[0,1]$ there exists $\bar{x} \in [0,1]$ such that $|f(\bar{x})| \geq |f(x)|$, $\forall x \in [0,1]$. Then clearly:
\begin{equation*}
|\mu_I+f(x)| \leq |\mu_I+f(\bar{x})| \leq |\mu_I+f(\bar{x})|(1+|x|),
\end{equation*}
thus, the condition is met for $g_0$. For $g_1= -1$ the conditions are met trivially.

\item Indeed $U(t)$ is closed and convex.

\item Indeed, $E$ and $Q$ are closed and $l=0$ is lower semicontinuous.

\item $\Lambda(t,x,w) = Cw+x$ is continuous in each variable, thus lower semicontinuous. Also $\Lambda(t,x,\cdot)$ is convex for each $(t,x)$, and there is a constant, namely $0$ such that for each $(t,x) \in Q$ and $w\in U(t)$ then:
\begin{equation*}
\Lambda(t,x,w) \geq 0.
\end{equation*}

\item Indeed, given that $E$ is bounded, then the projections of $E$ must be bounded also.

\item Let us consider $r=1$, $\alpha =c$, $\beta=0$. Then,
\begin{equation*}
\Lambda(t,x,w) \geq \alpha|w|^r + \beta.
\end{equation*}
\end{enumerate}

Finally, given that $f$ is locally Lipschitz-continuous in each initial condition considered, then there exists an admissible solution for \eqref{eq:dinamicaParticular}, and we conclude that there is solution for the optimal control problem. 

}\fi

\subsection{Solution of logistic equation and switching times}\label{app:solutionLogistic}

Our purpose is to obtain the solution of \eqref{eq:SIS-logistic-2}, for $w\in[0,\tildemuI]$ fixed. Let us consider the equation 
\begin{equation}\label{eq:log_gen}
\dot x = -Ax^2 + Bx + \mu_I-w,\qquad x(t_0)=x_0.
\end{equation}
This equation has the form  
\begin{equation}\label{eq:log_gen_fact}
\dot x = -A(x-r_w^{+})(x-r_w^{-}),
\end{equation}
with $r_w^{+}>r_w^{-}$ the roots of the equation $-Ax^2 + Bx + (\mu_I-w) =0$, namely, 
\begin{equation*}
r_w^{+} = \frac{B+\sqrt{ B^2 + 4A(\mu_I-w)}}{2A}\,,\quad r_w^{-} = \frac{B-\sqrt{ B^2 + 4A(\mu_I-w)}}{2A}.
\end{equation*}

The solution of \eqref{eq:log_gen_fact} is given by
\begin{equation}\label{eq:solucion_logistica2}
x^w(t;x_0,t_0) = \frac{r_w^{+}(x_0-r_w^{-}) - e^{-A(r_w^{+}-r_w^{-})(t-t_0)}r_w^{-}(x_0-r_w^{+}) }{ (x_0-r_w^{-}) - e^{-A(r_w^{+}-r_w^{-})(t-t_0)}(x_0-r_w^{+})  }, 
\end{equation}
or, in terms of the time, 
\begin{equation}\label{eq:tiempo_logistica2}
t-t_0 = \frac{1}{A(r_w^{+}-r_w^{-})}\log\left| \frac{(x^w(t;x_0,t_0)-r_w^{-})(x_0-r_w^{+}) }{(x^w(t;x_0,t_0)-r_w^{+})(x_0-r_w^{-})} \right|.
\end{equation}

Now, suppose that the final condition of an optimal trajectory $x(\cdot)$ is $x(T)=x_T$. Then, there exists a maximal interval $[\tau_S(x_T),T]$ such that $w(t)=0$ $a.e.\,t\in[\tau_S(x_T),T]$. This means that, during this part of the trajectory, the equilibria of \eqref{eq:log_gen_fact} are $r_{0}^{+}$ and $r_{0}^{-}$. $\tau_S(x_T)$ is the time of last switch, and it satisfies $x(\tau_S(x_T))=x_S(x_T)$. If we replace $(x_0,t_0)=(x_S(x_T),\tau_S(x_T))$, $(x,t)=(x_T,T)$ in \eqref{eq:tiempo_logistica2}, we get
\begin{equation}\label{eq:switch_time}
\tau_S(x_T)= T-\frac{1}{\sqrt{\Delta}}\log\left| \frac{(x_T-r_0^{-})(x_S(x_T)-r_0^{+}) }{(x_T-r_0^{+})(x_S(x_T)-r_0^{-})} \right|,
\end{equation}
with $x_S(x_T)$ given by \eqref{eq:xsxtxchico} or \eqref{eq:xsxtgrande}, according to the corresponding case. Notice that, thanks to Lemma \ref{lemma:xs_xt}, the argument of $\log|\cdot|$ in \eqref{eq:switch_time} is always nonnegative. This function allows to define the curve of the state and time of last switch in the $(x,t)-$state space, parametrized by the final condition $x_T$ of the respective optimal trajectory, as $\varphi(x_T)=(x_S(x_T),\tau_S(x_T))$,  for $x_T$ such that the curve is well defined. Since the only possible states for which there may be a last switch from $w=\tildemuI$ to $w=0$ belong to the set $(\rzeromenos,\xssup)$, we can reparametrize this switching curve in terms of the switching state $x_s$ as follows:

\begin{itemize}
\item Consider as initial condition $x_{s}=x_S(x_T)\in(\rzeromas,\xssup)$, and find $x_T(x_{s})$ the inverse function of $x_S(x_T)$, that solves \eqref{eq:def_xs_xt}, that is
\begin{equation*}
\lambda_{\inf}(x_{s},x_T(x_{s}))=C.
\end{equation*}
From the definition of $\lambda_{\inf}(\cdot)$ in \eqref{eq:def_lambda_inf}, we see that 
\begin{equation}\label{eq:xt_xs}
x_T(x_{s})=C(\mu_I+f(x_{s}))+x_{s} = -AC(x_{s}-\rzeromenos)(x_{s}-\rzeromas)+x_{s}.
\end{equation}
\item Replace $x_T=x_T(x_s)$ in \eqref{eq:switch_time}, and define
\begin{equation*}\label{eq:ts}
t_S(x_s): = \tau_S(x_T(x_s)) =  T-\frac{1}{\sqrt{\Delta}}\log\left| \frac{(x_T(x_s)-\rzeromenos)(x_s-\rzeromas)}{(x_T(x_s)-\rzeromas)(x_s-\rzeromenos)} \right|
\end{equation*}
which, using \eqref{eq:xt_xs}, can be written as
\begin{equation*}\label{eq:ts2}
t_S(x_s) = T-\frac{1}{\sqrt{\Delta}}\log\left| \frac{1-AC(x_s-\rzeromas)}{1-AC(x_s-\rzeromenos)} \right|.
\end{equation*}
\end{itemize}

For the case $\rbar<\xbar$, another switch can be identified (see Figure \ref{fig:diag_fase}). 
Remark \ref{rem:switches} shows that this switch is symmetrical to the point of last switch $x_s$ with respect to $\xbar$, satisfying
\begin{equation}\label{eq:xsigma_xs}
x_{\sigma}(x_s) = 2\xbar - x_s.
\end{equation}

Notice that $x_{\sigma}(x_s)$, as a first switching point related to the last switching point $x_s$, makes sense only if $x_s\in(\rmumas,\xssup)$. Indeed, the set $(\rmumas,\infty)$ is positively invariant under the dynamics \eqref{eq:log_gen} with control $w=\mu_I$. Thus, no point $(x_s,t_s)$ with $x_s<\rmumas$ can belong to the same trajectory of a point $(x_{\sigma}(x_s),t_{\sigma})$, with $x_{\sigma}(x_s)>\rmumas$. Then, the domain of $x_{\sigma}(x_s)$ is the interval $(\rmumas,\xssup)$.

Let us define, along with $x_{\sigma}(x_s)$, the corresponding time of switch $\tau_{\sigma}(x_s)$, obtained as function of $x_s$, replacing $(x_0,t_0)=(x_{\sigma}(x_s),\tau_{\sigma}(x_s))$, $(x,t)=(x_s,t_S(x_s))$, and $w=\tildemuI$ in \eqref{eq:tiempo_logistica2}:
\begin{equation}\label{eq:tau_sigma}
\tau_{\sigma}(x_s) = t_S(x_s)-\frac{1}{B}\log\left| \frac{x_s\left(x_{\sigma}(x_s)-\rmumas\right)}{\left(x_s-\rmumas\right)x_{\sigma}(x_s)} \right|.
\end{equation}

As previously done, we can reparametrize these functions in terms of the first switching point $x_{\sigma}$ as follows:
\begin{itemize}
\item From \eqref{eq:xsigma_xs}, the last switch $x_s$ written as function of the first switch $x_{\sigma}$ as
\begin{equation}\label{eq:xs_xsigma}
x_s(x_{\sigma}) = 2\xbar - x_{\sigma},\qquad x_{\sigma}\in(2\xbar-\xssup,2\xbar-\rmumas).
\end{equation}
\item Replacing \eqref{eq:xs_xsigma} in \eqref{eq:tau_sigma}, define
\begin{equation*}\label{eq:t_sigma_app}
t_{\sigma}(x_{\sigma}) := \tau_{\sigma}(x_s(x_{\sigma})) =  t_S(2\xbar-x_{\sigma})-\frac{1}{B}\log\left| \frac{(2\xbar-x_{\sigma})\left(x_{\sigma}-\rmumas\right)}{\left(2\xbar-x_{\sigma}-\rmumas\right)x_{\sigma}} \right|.
\end{equation*}
\end{itemize}

\subsection{Sensitivity of an ODE flow and hitting times}\label{app:sensitivity}

Consider $x(t;x_0,t_0)$ the solution of an autonomous ordinary differential equation, of the form
\begin{equation}\label{eq:ODE}
\dot x(t) = F(x(t)), \quad t>t_0; \quad x(t_0)=x_0,
\end{equation} 
with $F(\cdot)\in \mathcal C^1(\R)$. Following \cite{victor:ode}, if we suppose that \eqref{eq:ODE} defines a global flow $x(t;x_0,t_0)$, then for each $t\in\R$, the solution satisfies the partial differential equation
\begin{equation}\label{eq:PDE_ODE_F}
\left\{\quad\begin{split}
\frac{\partial x}{\partial t_0}(t;x_0,t_0) +  F(x_0)\frac{\partial x}{\partial x_0}(t;x_0,t_0) &= 0,\quad t_0<t;\\
x(t;x_0,t) &= x_0.
\end{split}\right.
\end{equation}

Since $x(\cdot)$ is a flow, generated by an autonomous system, it holds
\begin{equation}\label{eq:dflujo_dt0}
\frac{\partial x}{\partial t_0}(t;x_0,t_0) = -\frac{\partial x}{\partial t}(t;x_0,t_0) = -F(x(t;x_0,t_0)).
\end{equation}

Combining \eqref{eq:PDE_ODE_F} and \eqref{eq:dflujo_dt0}, if $F(x_0)\neq0$,
\begin{equation}\label{eq:dflujo_dx0}
\frac{\partial x}{\partial x_0}(t;x_0,t_0) = \frac{F(x(t;x_0,t_0))}{F(x_0)}.
\end{equation}

Now, suppose that we have a set $S$, defined as the zero level set of a $\mathcal C^1$ function $G:\R^2\rightarrow\R$, that is, $S=\{(x,t)\,|\, G(x,t)=0\}$, and define the first hitting time of $S$ as
\begin{equation*}\label{eq:hitting_time}
\hat t_S(x_0,t_0) := \inf\{t> t_0\,|\, (x(t;x_0,t_0),t)\in S \}. 
\end{equation*}
and the hitting state of $S$ as
\begin{equation}\label{eq:hitting_state}
\hat x_S(x_0,t_0) := x(\hat t_S(x_0,t_0);x_0,t_0).
\end{equation}

Then, from \cite{victor:ode}, $\hat t_S(\cdot)$ and $\hat x_S(\cdot)$ are differentiable at every $(x_0,t_0)\notin S$ such that 
\begin{equation*}
\frac{\partial G}{\partial t}(\hat x_S(x_0,t_0),\hat t_S(x_0,t_0)) + F(\hat x_S(x_0,t_0))\frac{\partial G}{\partial x}(\hat x_S(x_0,t_0),\hat t_S(x_0,t_0)) \neq0,
\end{equation*}
and, at such $(x_0,t_0)$, we have
\begin{equation}\label{eq:edp_tS}
\frac{\partial \hat t_S}{\partial t_0}(x_0,t_0)+F(x_0)\frac{\partial \hat t_S}{\partial x_0}(x_0,t_0) \,=\, 0,
\end{equation}
and 
\begin{equation}\label{eq:edp_xS}
\frac{\partial \hat x_S}{\partial t_0}(x_0,t_0)+F(x_0)\frac{\partial \hat x_S}{\partial x_0}(x_0,t_0) = 0.
\end{equation}

Notice that, differentiating \eqref{eq:hitting_state} with respect to $t_0$, and using \eqref{eq:dflujo_dt0}, we obtain
\begin{equation*}\label{eq:hitting_state_derivatives_t0}
\frac{\partial \hat x_S}{\partial t_0}(x_0,t_0) - F(\hat x_S(x_0,t_0))\frac{\partial \hat t_S}{\partial t_0}(x_0,t_0) 
\,=\, \left.\frac{\partial x}{\partial t_0}(t;x_0,t_0)\right|_{t=\hat t_S(x_0,t_0)} = -F(\hat x_S(x_0,t_0)).
\end{equation*}
Differentiating \eqref{eq:hitting_state} with respect to $x_0$, and using \eqref{eq:dflujo_dx0}, if $F(x_0)\neq0$, we obtain
\begin{equation*}\label{eq:hitting_state_derivatives_x0}
\frac{\partial \hat x_S}{\partial x_0}(x_0,t_0) - F(\hat x_S(x_0,t_0))\frac{\partial \hat t_S}{\partial x_0}(x_0,t_0) 
\,=\, \left.\frac{\partial x}{\partial x_0}(t;x_0,t_0)\right|_{t=\hat t_S(x_0,t_0)} = \frac{F(\hat x_S(x_0,t_0))}{F(x_0)}.
\end{equation*}

\subsection{Cost functionals}\label{app:valueFunction}

Consider a fixed control value $w\in[0,\tildemuI]$ for the dynamic \eqref{eq:log_gen}. The associated cost function in the interval $[t_i,t_f]$ for a trajectory starting from the initial condition $x_i$ is given by the cost functional
\begin{equation*}\label{eq:Jw_a}
J^{w}(x_i,t_i,t_f) \,:=\, \int_{t_i}^{t_f}x^{w}(t;x_i,t_i)dt + Cw(t_f-t_i),
\end{equation*}
where $x^w(\cdot;x_i,t_i)$ is the solution of \eqref{eq:log_gen}, with constant control $w(t)\equiv w$, such that $x^w(t_i;x_i,t_i)=x_i$, given explicitly in \eqref{eq:solucion_logistica}. Simple integration leads to
\begin{equation*}\label{int_x}
\int_{t_1}^{t_2}x^w(t;x_i,t_i)dt = \rwmas(t_2-t_1)+\frac{1}{A}\log\left| \frac{(x_i-r_w^{-})-(x_i-r_w^{+})e^{-A(r_w^{+}-r_w^{-})(t_2-t_i)}}{(x_i-r_w^{-})-(x_i-r_w^{+})e^{-A(r_w^{+}-r_w^{-})(t_1-t_i)}} \right|.
\end{equation*}
From this expression, we obtain \eqref{eq:Jw_expl}, and the derivatives in \eqref{eq:partialJ_cuerpo}.

\end{document}